\documentclass[journal,letterpaper]{IEEEtran}

\usepackage[cmex10]{amsmath}
\usepackage{amsthm,aliascnt,hyperref,bm}
\usepackage{amssymb,balance}
\usepackage{url}

\usepackage{xr}
\makeatletter

\newcommand\Autoref[1]{\@first@ref#1,@}
\def\@throw@dot#1.#2@{#1}% discard everything after the dot
\def\@set@refname#1{%    % set \@refname to autoefname+s using \getrefbykeydefault
	\edef\@tmp{\getrefbykeydefault{#1}{anchor}{}}%
	\def\@refname{\@nameuse{\expandafter\@throw@dot\@tmp.@autorefname}s}%
}
\def\@first@ref#1,#2{%
	\ifx#2@\autoref{#1}\let\@nextref\@gobble% only one ref, revert to normal \autoref
	\else%
	\@set@refname{#1}%  set \@refname to autoref name
	\@refname~\ref{#1}% add autoefname and first reference
	\let\@nextref\@next@ref% push processing to \@next@ref
	\fi%
	\@nextref#2%
}
\def\@next@ref#1,#2{%
	\ifx#2@ and~\ref{#1}\let\@nextref\@gobble% at end: print and+\ref and stop
	\else, \ref{#1}% print  ,+\ref and continue
	\fi%
	\@nextref#2%
}

\makeatother

\newcommand{\mynewtheorem}[2]{
  \newaliascnt{#1}{dummy}
  \newtheorem{#1}[#1]{#2}
  \aliascntresetthe{#1}
  \expandafter\def\csname #1autorefname\endcsname{#2}
}

\theoremstyle{plain}
  \mynewtheorem{theorem}{Theorem}
  \mynewtheorem{proposition}{Proposition}
  \mynewtheorem{corollary}{Corollary}
  \mynewtheorem{lemma}{Lemma}
  \mynewtheorem{primal}{Primal}
  \mynewtheorem{dual}{Dual}
  \mynewtheorem{example}{Example}

\theoremstyle{definition}
  \mynewtheorem{definition}{Definition}
  \mynewtheorem{problem}{Problem}
\theoremstyle{remark}
  \mynewtheorem{remark}{Remark}

\newcommand{\appendixref}[1]{\hyperref[#1]{Appendix}}

\usepackage{amsfonts}
\usepackage{array}
\usepackage{verbatim}
\usepackage{enumerate}
\usepackage{tikz}
\usetikzlibrary{arrows,shapes,calc,through,backgrounds,intersections,positioning}

\tikzset{
    right angle quadrant/.code={
        \pgfmathsetmacro\quadranta{{1,1,-1,-1}[#1-1]}     % Arrays for selecting quadrant
        \pgfmathsetmacro\quadrantb{{1,-1,-1,1}[#1-1]}},
    right angle quadrant=1, % Make sure it is set, even if not called explicitly
    right angle length/.code={\def\rightanglelength{#1}},   % Length of symbol
    right angle length=2ex, % Make sure it is set...
    right angle symbol/.style n args={3}{
        insert path={
            let \p0 = ($(#1)!(#3)!(#2)$) in     % Intersection
                let \p1 = ($(\p0)!\quadranta*\rightanglelength!(#3)$), % Point on base line
                \p2 = ($(\p0)!\quadrantb*\rightanglelength!(#2)$) in % Point on perpendicular line
                let \p3 = ($(\p1)+(\p2)-(\p0)$) in  % Corner point of symbol
            (\p1) -- (\p3) -- (\p2)
        }
    }
}
\usepackage{color}

  \DeclareMathOperator*{\relint}{relint}
  \DeclareMathOperator*{\Conv}{Conv}
  \DeclareMathOperator*{\affinespan}{aff}

\newcommand{\icl}[1]{{\color{black}{#1}}}

\usepackage{graphicx}
\usepackage{epstopdf}

\begin{document}
\title{Stability and instability in saddle point dynamics \\Part II: The subgradient method}
\author{Thomas Holding\thanks{\color{black}Thomas Holding is with the \icl{Department of Mathematics, Imperial College London, United Kingdom; \tt\small t.holding@imperial.ac.uk}} and Ioannis~Lestas\thanks{Ioannis Lestas is with the Department of Engineering, University of Cambridge, Trumpington Street, Cambridge, CB2 1PZ, United Kingdom.
%, and with the Cyprus University of Technology;
 \tt\small icl20@cam.ac.uk}}

% The paper headers
\markboth{}%
{}

\maketitle
\begin{abstract}\color{black}
In part I we considered the problem of convergence to a saddle point of a concave-convex function via gradient dynamics and an exact characterization was given to their asymptotic behaviour.
%of such dynamics and it was shown that despite being nonlinear their limiting trajectories  are %always solutions to an explicit linear ODE.
In part II we consider a general class of subgradient dynamics that provide a restriction in an arbitrary convex domain.  %\icl{and have an equilibrium point}.
We show that despite the nonlinear and non-smooth character of these dynamics their $\omega$-limit set is comprised of solutions to only linear ODEs. In particular, we show that the latter are solutions to subgradient dynamics on affine subspaces which is a smooth class of dynamics the asymptotic properties of which have been exactly characterized in part I. Various convergence criteria are formulated using these results and several examples and applications are also discussed throughout the manuscript.
%It is known that for a strictly concave-convex function, the gradient method introduced by Arrow, Hurwicz and Uzawa \cite{arrow}, has guaranteed global convergence to its saddle point.  Nevertheless, there are classes of problems where the function considered is not {strictly} concave-convex, in which case convergence to a saddle point is not guaranteed.
%
%In the paper we provide a characterization of the asymptotic behaviour of the gradient method, in the general case where this is applied to a general concave-convex function. We prove that for any initial conditions  the gradient method is guaranteed to converge to a trajectory described by an explicit linear ODE. We further show that this result has a natural extension to subgradient methods, where the dynamics are constrained in a prescribed convex set. The results are used to provide simple characterizations of the limiting solutions for special classes of optimization problems, and modifications of the problem so as to avoid oscillations are also discussed. {\color{purple}MENTION EXAMPLES?}
\end{abstract}


\IEEEpeerreviewmaketitle

%
%{\color{purple}
%%Does the title still fit? Should it be changed to reflect the new results? Or at least made %shorter and punchier?
%Abstract/title should also be updated with new results.}

\section{Introduction}
\IEEEPARstart{I}{n} \cite{Holding-Lestas-gradient-method-Part-I} we studied the asymptotic behaviour of the gradient method when this is applied on a general concave-convex function in an unconstrained domain, and provided an exact characterization to its limiting solutions. Nevertheless,
%In particular, the limiting solutions can in general be oscillatory when the function is not strictly concave-convex, we showed that despite the nonlinearity of the dynamics, these always satisfy a linear ODE that was explicitly characterized.
in many applications, such as primal/dual algorithms in optimization problems,
it becomes necessary to constrain the system states in a prescribed convex set, e.g. positivity constraints on Lagrange multipliers or constraints on physical quantities like data flow, and prices/commodities in economics \cite{Hurwicz}, \cite{KMT}, \cite{SrikantB}, \cite{Paganini}.
%notably in relaxing inequality constraints via Lagrange multipliers in optimization, it is %necessary to ensure positivity of the dual variables. More generally, it may be required that %the dynamics are constrained to lie in a prescribed convex set.
The subgradient method is used in such cases, which is a version of
%In these cases,
the gradient method %\eqref{gradmethod-fullspace}
with a projection term in the vector field additionally included,
so as to ensure that the trajectories do not leave the desired set.
%is modified by the addition of a discontinuous convex projection term, giving the subgradient %method \eqref{gradmethod-convex-domain}.

In discrete time, there is an extensive literature on the subgradient method, via its application in optimization problems (see e.g. \cite{Nedic-Osdaglar-subgradient-method-discrete-time}). However, in many applications, for example power networks %\cite{Power1}, \cite{Power2}, \cite{Power3}, \cite{Power3b},
\cite{Power1, Power2, Power3, Power3b, devane2016distributed, kasis2017stability, stegink2017unifying, li2016connecting, mallada2017optimal}
  and classes of data network problems \cite{KMT},  \cite{SrikantB}, \cite{Paganini}, \cite{low1999optimization} continuous time models are considered. It is thus important to have a good understanding of the subgradient dynamics in a continuous time setting, which could also facilitate analysis and design by establishing links with other more abstract results in dynamical systems theory.

A main complication in the study of the subgradient method arises from the fact the this is a {\em non-smooth} {\color{black}system, i.e. a nonlinear ODE with} a discontinuous vector field due to the projections involved.
%The discontinuity in the vector field in the subgradient method
This prohibits the direct application of classical Lyapunov or LaSalle theorems (e.g. \cite{khalil}), which is reflected in the direct approach used by Arrow, Hurwicz and Uzawa in \cite{arrow} that avoids the use of such tools.
\icl{It has been identified from an early stage that the right-hand side of subgradient dynamics is monotone \cite{rockafellar1971saddle}, which is a property that can facilitate their analysis \cite{goebel2017stability}. This has been exploited to derive convergence results
that %Convergence results that have been derived
have primarily relied on appropriate strictness in the concave-convex property of the saddle function. The work in \cite{venets1985continuous} showed that such strictness is sufficient at the saddle points and for one of the two sets of variables of the concave-convex function,
 %convexity/concavity at the saddle point is sufficient
 with an extension to non-smooth functions given in \cite{goebel2017stability}.

Various recent studies have also used tools from the analysis of hybrid and discontinuous systems to deduce related convergence results.}
%with the results in relaxing this to deviationin one of the two sets of variables for %deviations from the saddle points. }
%
%More recently,
The work of Feijer and Paganini \cite{Paganini}
%{\color{black} [named here to mitigate having to say they're wrong in the following sentence!]}
%unified the previously ad-hoc and application focused analysis of primal dual gradient dynamics %in network optimisation, and
proposed that the switching in the dynamics be interpreted in the framework of hybrid automata, using the \icl{invariance principle derived in \cite{Hybrid}.} However, as \icl{pointed} out in \cite{Cherukuri-primal-dual}, there are cases where the assumptions required in \cite{Hybrid} do not hold. In \cite{Cherukuri-primal-dual}, the \icl{invariance} principle for discontinuous Carath\'eodory systems is applied to prove convergence of the subgradient method under positivity constraints and the assumption of strict concavity. {\color{black}Further results on the asymptotic properties of the subgradient method under positivity constraints where derived in \cite{cherukuri2016role} where global convergence was also shown under a condition of local strict concavity-convexity.} In \cite{Richert-Cortes-Robust-distributed-linear-programming} the subgradient method is used to solve linear programs with inequality constraints. In general, proving convergence for the subgradient method even in simple cases, is a non-trivial problem that requires \color{black}{the non-smooth character of the system to be explicitly addressed.} %to explicitly address the non-smooth character of the system.

Our aim in this paper is to provide a framework of results that allow one to study the asymptotic behaviour of the subgradient method %\eqref{gradmethod-convex-domain}
%with %\emph{
%smooth analysis as opposed to %\emph{
%non-smooth analysis.
%{\color{black}In particular, we consider the subgradient method
{in a general setting, where the trajectories are constrained to an arbitrary convex domain, and the concave-convex function considered is not necessarily strictly concave-convex. One of our main results is to show that despite the nonlinear and non smooth character of the subgradient dynamics, their limiting behaviour \icl{when an equilibrium point exists}, are solutions to explicit {\em linear} differential equations.
%is given by  the solutions of one of an explicit family of {\color{black}\emph{linear}} %differential {\color{black}equations.}}

%is that the limiting behaviour
%of the subgradient method %constrained to an arbitrary convex domain
%is given by the solutions of one of an explicit family of {\color{black}\emph{linear}} differential {\color{black}equations.}}
%{\color{black}This result hence reduces the complications associate with non-smooth analysis and %allows classical Lyapunov and LaSalle type
%stability tools to be used.}
%With this result, proving convergence of the subgradient method may be done with standard %Lyapunov and LaSalle type stability tools, reducing the barrier to obtaining rigorous %convergence proofs in applications.

{\color{black}{\color{black}In particular, we show} that these {\color{black}linear} ODEs
%associated with the limiting behaviour of the subgradient method on an arbitrary convex domain,
are {\color{black}limiting solutions of} subgradient dyanmics on an {\em affine subspace}, which is a class of dynamics that fit within the framework studied in Part~I \cite{Holding-Lestas-gradient-method-Part-I}. %The asymptotic properties of
{\color{black}These} dynamics can therefore be exactly characterized, thus allowing to prove convergence to a saddle point for broad classes of problems.

%The subgradient method is often applied to distributed network optimisation, as it leads to decentralized update rules. Various modifications of the Lagrangian are also often proposed in the literature so as to provide convergence guarantees, while maintaining the localised structure of the dynamics.
%%it is common to modify the Lagrangian in such a way that oscillations are avoided and %convergence achieved, but the localised structure is maintained (see \cite{Paganini} for an %overview).
The results in this paper are illustrated by means of examples that demonstrate also the complications in the dynamic behaviour of the subgradient method relative to the unconstrained gradient method. We also apply our results to modification schemes in network optimization, that provide convergence guarantees while maintaining a decentralized structure in the {\color{black}dynamics.}}
%We show how to recover proofs of convergence of the subgradient method for some such %modification schemes, and we give generalisations of the proofs to case arbitrary convex %constraint sets.
  %Finally, we discuss} an application to the problem of multi-path congestion control (e.g. \cite{Voice2007,KMT,Lestas-Routing-CDC2004,Lin-Shroff}), where we prove convergence of a modification scheme, that achieves optimality without requiring any additional information transfer. % (this is associated with a Lagrangian that is not strictly concave-convex).

%{\color{black}The derivations in the paper are based on an interpalay betwen}

{\color{black}The methodology used for the derivations in the paper is also of independent technical interest. In particular, the notion of a face of a convex set is used to characterize the ODEs associated with the limiting behaviour of the subgradient dynamics. Furthermore, some more abstract results on corresponding semi-flows have been used to address the complications associated with the non-smooth character of subgradient dynamics.}

The paper is structured as follows. {\color{black}Section \ref{sec:Preliminaries} provides preliminaries from convex analysis and dynamical systems {\color{black}theory} that will be used within the paper.} The problem formulation is given in section \ref{sec:formulation} and the main results are presented in section \ref{sec:Main}, where various examples that illustrate those are also discussed. {\color{black}Applications to modification methods in network optimization are given in section \ref{sec:modification-methods}.
 The proofs of the results are given in appendices \ref{sec:proofs} and \ref{sec:proofs-of-examples} and an application to the problem of multipath routing is discussed in section~\ref{subsec:examples:multi-path-routing}.}
\section{Preliminaries}
\label{sec:Preliminaries}
We use the same notation and definitions {\color{black}as in} part I of this work \cite{Holding-Lestas-gradient-method-Part-I} and we refer the reader to the preliminaries section therein. {\color{black}The notions below from convex analysis and analysis of dynamical systems will additionally be used throughout the paper.}

\subsection{Convex analysis}
%{\color{black}repeat from Part I some of the convex geometry definitions frequently used, %$N_K(z), N_K, P_K(z)$.}
{\color{black}We recall first for convenience the following notions defined in part I \cite{Holding-Lestas-gradient-method-Part-I} that will be frequently used in this manuscript. For a closed convex set $K\subseteq\mathbb{R}^n$ and $\mathbf{z}\in\mathbb{R}^n$, we denote the normal cone to $K$ through $\mathbf{z}$ as $N_K(\mathbf{z})$. When $K$ is an affine space $N_K(\mathbf{z})$ is independent of $\mathbf{z}\in K$ and is denoted $N_K$. If $K$ is in addition non-empty, then we denote the projection of $\mathbf{z}$ onto $K$ as $\mathbf{P}_{K}(\mathbf{z})$. {\color{black}Also for vectors $x,y\in\mathbb{R}^n$, $d(x,y)$ denotes the Euclidean metric and $|x|$ the Euclidean norm.
}
}

\subsubsection{Concave-convex functions and saddle points}
For a function $\varphi$ that is concave-convex on $\mathbb{R}^{n+m}$ the (standard) notion of a saddle point was given in part I \cite{Holding-Lestas-gradient-method-Part-I}. We now consider $\varphi$ restricted to a non-empty closed convex set $K\subseteq\mathbb{R}^{n+m}$, in which case the notion of saddle point needs to be modified to incorporate the constraints.
\begin{definition}[Restricted saddle point]
Let $K\subseteq\mathbb{R}^{n+m}$ be non-empty closed and convex. For a concave-convex function $\varphi:K\to\mathbb{R}$, we say that $(\bar{x},\bar{y})\in K$ is a \textit{$K$-restricted saddle point} of $\varphi$ if for all $x\in\mathbb{R}^n$ and $y\in\mathbb{R}^m$ with $(x,\bar{y}),(\bar{x},y)\in K$ we have the inequality $\varphi(x,\bar{y})\le\varphi(\bar{x},\bar{y})\le\varphi(\bar{x},y)$.
\end{definition}
If in addition $\varphi\in C^1$ then $\mathbf{\bar{z}}=(\bar{x},\bar{y})\in K$ is a $K$-restricted saddle point \icl{if} the vector of partial derivatives %$\begin{bmatrix}\varphi_x(\mathbf{\bar{z}})&\!\!\! %-\varphi_y(\mathbf{\bar{z}})\end{bmatrix}^T$
 $(\varphi_x(\mathbf{\bar{z}}), \varphi_y(\mathbf{\bar{z}}))$
lies in the normal cone $N_K(\mathbf{\bar{z}})$.

Any $K$-restricted saddle point in the interior of $K$ is also a saddle point. If $C\subseteq K$ is closed and convex and $\mathbf{\bar{z}}\in C$ is a $K$-restricted saddle point, then $\mathbf{\bar{z}}$ is also a $C$-restricted saddle point.

\icl{It is common in the literature (e.g. \cite{rockafellar1970convex}) for the set $K$ to be the cartesian product of two convex sets, with $x,y$ taking values in each of these two sets respectively, i.e.
\begin{subequations}\label{eq:K_phi}
\begin{align}\label{set:K}
K&=K_x\times K_y, \ K_x\subset\mathbb{R}^n, \  K_y\subset\mathbb{R}^m, \\
&\qquad \qquad \qquad K_x,K_y \ \text{convex closed sets} \nonumber\\  \varphi&:K_x\times K_y\rightarrow\mathbb{R}, \ \varphi \text{ is concave-convex}
\end{align}
\end{subequations}

It should be noted that in this case if $\varphi\in C^1$, then $\mathbf{\bar{z}}=(\bar{x},\bar{y})\in K$ is a $K$-restricted saddle point {if and only if} the vector of partial derivatives $(\varphi_x(\mathbf{\bar{z}}), \varphi_y(\mathbf{\bar{z}}))$ lies in the normal cone $N_K(\mathbf{\bar{z}})$.

}

\icl{In general it} does not hold that if $\varphi:\mathbb{R}^{n+m}\to\mathbb{R}$ has a saddle point, and $K$ is closed convex and non-empty, then $\varphi$ has a $K$-restricted saddle {\color{black}point (an explicit} example {\color{black}illustrating} this is given later in \autoref{example:disappearance-of-saddle-points}(ii)). In this manuscript we will only consider cases where at least one $K$-restricted saddle point exists, leaving the problem of showing existence to the specific application.

\subsubsection{Concave programming}\label{sec:application-to-concave-optimization}
Concave programming (see e.g. \cite{Boyd}) is concerned with the study of optimization problems of the form
\begin{equation}\label{primal}
\max_{x\in C,g(x)\ge0} U(x)
\end{equation}
where $U:\mathbb{R}^n\to\mathbb{R}$, $g:\mathbb{R}^n\to\mathbb{R}^m$
% and $h:\mathbb{R}^n\to\mathbb{R}^{m'}$
are concave functions and $C\subseteq \mathbb{R}^n$ is non-empty closed and convex.
{\color{black}Under some mild assumptions, the solutions to such problems are saddle points of the Lagrangian}
% This is associated with the Lagrangian
\begin{equation}\label{lagrangian-form}
\varphi(x,y)=U(x)+y^T g(x)
\end{equation}
where $y\in\mathbb{R}^m_+$ are the Lagrange multipliers. {\color{black}This is stated in the Theorem below}.
% Let $K=C\times\mathbb{R}_+^{m}$, then under Slater's condition \eqref{slater}, finding a $K$-restricted saddle point $(\bar{x},\bar{y})$ of \eqref{lagrangian-form} is equivalent to solving \eqref{primal}.
\begin{theorem}
Let $g$ be concave and Slater's condition hold, i.e.
\begin{equation}\label{slater}
 \exists x'\in\relint C\text{ with }g(x')>0.
\end{equation}
Then $\bar{x}$ is an optimum of \eqref{primal} if and only if $\exists\bar{y}$ with $(\bar{x},\bar{y})$ a $C\times\mathbb{R}^{m}_+$-restricted saddle point of \eqref{lagrangian-form}.
\end{theorem}
The min-max optimization problem associated finding a {\color{black}$C\times\mathbb{R}^{m}_+$-restricted} saddle point of \eqref{lagrangian-form} is the dual problem of \eqref{primal}.

%If the optimization problem is of the simpler form
%\begin{equation}\label{primal-simple}
%\max_{x\in\mathbb{R}^n,g(x)=0}\,U(x)
%\end{equation}
%then the associated Lagrange multipliers in \eqref{lagrangian-form} are no longer constrained to be non-negative, and finding a saddle point of \eqref{lagrangian-form} is equivalent to solving \eqref{primal-simple}.
\subsubsection{Faces of convex sets}
Some of the main results of this manuscript refer to faces of a convex set. We refer the reader to \cite[Chap. 1.8.]{Handbook-convex-geometry} for further discussion of such topics.
\begin{definition}[Face of a convex set]\label{def:face}
Given a non-empty closed convex set $K$, a face $F$ of $K$ is a subset of $K$ that has both the following properties:
\begin{enumerate}[(i)]
\item $F$ is convex.
\item For any line segment $L\subseteq K$, if $(\relint L)\cap F\ne\emptyset$ then $L\subseteq F$.
\end{enumerate}
\end{definition}
For the readers convenience we recall some standard properties of faces:
\begin{enumerate}[(a)]
\item The intersection of two faces of $K$ is a face of $K$.
\item The empty set and $K$ itself are both faces of $K$. If a face $F$ is neither $\emptyset$ or $K$ it is called a proper face.
\item If $F$ is a face of $K$ and $F'$ is a face of $F$, then $F'$ is a face of $K$.
\item For a face $F$ of $K$, the normal cone $N_K(\mathbf{z})$ is independent of the choice of $\mathbf{z}\in\relint(F)$. In these cases we drop the $\mathbf{z}$ dependence and write it as $N_F$.
\item $K$ may be written as the disjoint union:
\begin{equation}
K=\bigcup\{\relint F:F\text{ is a face of }K\}.
\end{equation}
\end{enumerate}
Property (a) above leads to the following definition.
\begin{definition}[Minimal face containing a set]\label{def:minimal-face}
For a convex set $K$ and a subset $A\subseteq K$ we define the \emph{minimal face containing $A$} as
\begin{equation*}
\bigcap\{F:F \text{ is a face of }K\text{ and }A\subseteq F\}
\end{equation*}
which is a face by property (a) above.
\end{definition}

\subsection{Dynamical systems}
\begin{definition}[Flows and semi-flows]\label{def:semiflow}
A triple $(\phi,X,\rho)$ is a flow (resp. semi-flow) if $(X,\rho)$ is a metric space, $\phi$ is a continuous map from $\mathbb{R}\times X$ (resp. $\mathbb{R}_+\times X$) to $X$ which satisfies the two properties
\begin{enumerate}[(i)]
\item For all $x\in X$, $\phi(0,x)=x$.
\item For all $x\in X$, $t,s\in\mathbb{R}$ (resp. $\mathbb{R}_+$),
\begin{equation}
\phi(t+s,x)=\phi(t,\phi(s,x)).
\end{equation}
\end{enumerate}
When there is no confusion over which (semi)-flow is meant, we shall denote $\phi(t,x(0))$ as $x(t)$. For sets $A\subseteq\mathbb{R}$ (resp. $\mathbb{R}_+$) and $B\subseteq X$ we define $\phi(A,B)=\{\phi(t,x):t\in A,x\in B\}$.
\end{definition}
\icl{We say that a trajectory $x(t)$ of a semi-flow converges to a trajectory $y(t)$ of the semi-flow if $\rho(x(t)-y(t))\to0$ as $t\to\infty$.}

\begin{definition}[$\omega$-limit set]
Given a semi-flow $(\phi,X,\rho)$ its $\omega$-limit set is the set %we denote the set of $\omega$-limit points of trajectories as
\begin{equation}
\Omega(\phi,X,\rho)=\bigcup_{x\in X}\bigcap_{t\ge0}\overline{\phi([t,\infty),x)}.
\end{equation}
where $\overline{A}$ denotes the closure of $A\subseteq X$ in $(X,\rho)$.
\end{definition}
\icl{
\begin{remark}
Note that a point $y$ is in the $\omega$-limit set of semi-flow $(\phi,X,\rho)$ if there exists a sequence $(t_n)_{n=\mathbb{N}}$  in $\mathbb{R}$ such that $\lim_{n\rightarrow\infty} t_n = \infty$ and $\lim_{n\rightarrow\infty} \phi(t_n,x)= y$ for some $x\in X$. Point $y$ is called an $\omega$-limit point of the solution $\phi(t,x)$.
\end{remark}
}
\begin{definition}[Invariant sets]
For a semi-flow $(\phi,X,\rho)$ we say that a set $A\subseteq X$ is positively invariant if $\phi(\mathbb{R}_+,A)\subseteq A$. If $\phi$ is also a flow we say that $A$ is negatively invariant if $\phi((-\infty,0],A)\subseteq A$. If $\phi(t,A)=A$ for all $t\in\mathbb{R}$ then we say $A$ is invariant.
\end{definition}
\begin{definition}[Sub-(semi)-flow]\label{def:sub-flow}
For a flow (resp. semi-flow) $(\phi,X,\rho)$ and an invariant (resp. positively invariant) set $A\subseteq X$ we obtain the sub-flow (resp. sub-semi-flow) by restricting $\phi(t,x)$ to act on $x\in A$ and denote it as $(\phi,A,\rho)$.
\end{definition}

\begin{definition}[Global convergence]\label{def:GlobConv}
We say that a (semi)-flow $(\phi,X,\rho)$ is \textit{globally convergent}, if for all initial conditions $x\in X$, the trajectory $\phi(t,x)$ converges to %the set of
\icl{an equilibrium point} of $(\phi,X,\rho)$ as $t\to\infty$.
%, i.e.
%\begin{equation*}
%\inf\{d(\phi(t,x),y):y\text{ an equilibrium point}\}\to 0\text{ as }t\to\infty.
%\end{equation*}
\end{definition}

In part I of this work much of the analysis relied on a specific {\color{black}form of stability, linked to incremental stability,} which we reproduce below for the convenience of the reader.
\begin{definition}[Pathwise stability]\label{def:pathwise-stability}
We say that a semi-flow $(\phi,X,\rho)$ is pathwise {\color{black}stable}
%\footnote{{\color{purple} DUPLICATED IN PART I! Notions similar to this have been studied %before by many different mathematical communities under different names, such as %\textit{monotone}, \textit{contractive} or \textit{dissipative} systems. As these terms are %already used in the control community for different concepts, we shall not use them to avoid %confusion.}}
if for any two trajectories $x(t),x'(t)$ the distance $\rho(x(t),x'(t))$ is non-increasing in time.
%Let $C^1\ni \mathbf{f}:\mathbb{R}^n\to\mathbb{R}^n$ and define a dynamical system $\dot{\mathbf{z}}=\mathbf{f}(\mathbf{z})$.
%\begin{equation}\label{ODE-for-def-of-incremental-stability}
%\dot{\mathbf{z}}=\mathbf{f}(\mathbf{z})
%\end{equation}
%Then we say the system is \textit{pathwise stable} if for any two solutions $\mathbf{z}(t),\mathbf{z}'(t)$ %\eqref{ODE-for-def-of-incremental-stability}
%the distance $d(\mathbf{z}(t),\mathbf{z}'(t))$,
%\begin{equation}
%d(\mathbf{z}(t),\mathbf{z}'(t))
%\end{equation}
%is non-increasing in time.
\end{definition}
{\color{black}As it will be discussed in the paper, the  $\omega$-limit set of pathwise stable semiflows, is comprised of semiflows of the class defined below.}
 %the following special class of (semi)-flows.
\begin{definition}[(Semi)-Flow of isometries]\label{def:flow-of-isometries}
We say that a (semi)-flow $(\phi,X,\rho)$ is a (semi)-flow of isometries if for every $t\in \mathbb{R}$ (resp. $\mathbb{R}_+$), the function $\phi(t,\cdot):X\to X$ is an isometry, i.e. for all $x,y\in X$ it holds that $\rho(\phi(t,x),\phi(t,y))=\rho(x,y)$.
\end{definition}
{\color{black} Finally, we will need the notion} of Carath\'eodory solutions of differential equations.
\begin{definition}[Carath\'eodory solution]\label{def:Caratheodory-solution}
We say that a trajectory $\mathbf{z}(t)$ is a \textit{Carath\'eodory solution} to a differential equation $\dot{\mathbf{z}}=\mathbf{f}(\mathbf{z})$, if $\mathbf{z}$ is an absolutely continuous function of $t$, and for almost all times $t$, the derivative $\dot{\mathbf{z}}(t)$ exists and is equal to $\mathbf{f}(\mathbf{z}(t))$.
\end{definition}
%Note that we do not require that $\mathbf{f}$ satisfies the assumptions of the Carath\'eodory existence theorem.
\section{Problem formulation}
\label{sec:formulation}
The main object of study in this work is the \emph{subgradient method} on an arbitrary concave-convex function in $C^2$ and an arbitrary convex domain $K$. We first recall the definition of the \emph{gradient method}, which is studied in part I of this work \cite{Holding-Lestas-gradient-method-Part-I}.

%The first focus of this paper is the \emph{gradient method} which is defined below.
\begin{definition}[Gradient method]
\label{gradmethod-definition}
Given $\varphi$ a $C^2$ concave-convex function on $\mathbb{R}^{n+m}$, we define the \textit{gradient method} as the flow on $(\mathbb{R}^{n+m},d)$ generated by the differential equation
\begin{equation}
\label{gradmethod-fullspace}
\begin{aligned}
%\dot{x}&=\varphi_x,&\quad&\quad&\dot{y}&=-\varphi_y.
\dot{x}&=\varphi_x\\
\dot{y}&=-\varphi_y.
\end{aligned}
\end{equation}
\end{definition}
%It is clear that the saddle points of $\varphi$ are exactly the equilibrium points of \eqref{gradmethod-fullspace}. In \autoref{sec:the-addition-of-constant-gains} we consider the addition of constant gains to the gradient method.
%\subsection{The subgradient method on convex domains}\label{subsec:preliminaries:subgradient-method}
%It is common in applications for the concave-convex function $\varphi$ to either be only defined on some convex subset $K\subseteq\mathbb{R}^{n+m}$ or to only be concave-convex on this subset. In these cases the subgradient method is used to prevent trajectories from leaving $K$.

The \emph{subgradient method} is obtained by restricting the gradient method to a convex set $K$ by the addition of a projection term to the differential equation \eqref{gradmethod-fullspace}.
\begin{definition}[Subgradient method]\label{def:subgradient-method}
Given a non-empty closed convex set $K\subseteq\mathbb{R}^{n+m}$ and a $C^2$ function $\varphi$ that is concave-convex on $K$, we define the \textit{subgradient method on $K$} as a semi-flow on $(K,d)$ consisting of Carath\'eodory {\color{black}solutions of}
\begin{equation}\label{gradmethod-convex-domain}
\begin{aligned}
\dot{\mathbf{z}}&=\mathbf{f}(\mathbf{z})-\mathbf{P}_{N_K(\mathbf{z})}(\mathbf{f}(\mathbf{z}))\\
\mathbf{f}(\mathbf{z})&=\icl{\begin{bmatrix}
                        \varphi_x\\ %&\!\!\!
			-\varphi_y
                       \end{bmatrix}}.
%\!\!\mathbf{\dot{z}}=\mathbf{f}(\mathbf{z})-\mathbf{P}_{N_K(\mathbf{z})}(\mathbf{f}(\mathbf{z}))
%\quad\!\text{where }\mathbf{f}(\mathbf{z})=\begin{bmatrix}
%                        \varphi_x&\!\!\!
%			-\varphi_y
%                       \end{bmatrix}^T.
\end{aligned}
\end{equation}
%where the notion of Carath\'eodory solution to a differential equation is defined in %\autoref{def:Caratheodory-solution}.
\end{definition}
%As explained in \cite[Appendix A]{Holding-Lestas-gradient-method-Part-I}, all the results of this paper can be translated to the subgradient method with gains
%\footnote{\color{black}By subgradient method with gains we refer to the ODE in \eqref{gradmethod-convex-domain} with constants multiplying each of the components of $\mathbf{f}(\mathbf{z})$  and the projection term appropriately modified to ensure that that the trajectories remain in the convex set $K$ (see details in \cite[Appendix A]{Holding-Lestas-gradient-method-Part-I}).}
by a transformation of coordinates. %{\color{black}PERHAPS REMOVE THIS PARAGRAPH}

\icl{
The equilibrium points of the subgradient method on $K$ are $K$-restricted saddle points. If in addition the set $K$ is the cartesian product of two convex sets
%as and function $\phi$ are
as in \eqref{eq:K_phi} then the set of equilibrium points of the subgradient method on $K$ is equal to the set of $K$-restricted saddle points.}

\begin{remark}\label{rem:non-smoothness-of-subgradient-method}
For (non-affine) convex sets $K$ the subgradient method \eqref{gradmethod-convex-domain} is a \emph{non-smooth} system. The vector field is discontinuous due to the convex projection term, independently of the regularity of the function $\varphi$ or of the boundary of $K$. This is in contrast to the gradient method \eqref{gradmethod-fullspace}, which is a \emph{smooth} system, as it inherits the regularity of the function $\varphi$.
\end{remark}

We briefly summarise the contributions of this work in the bullet points below.
\begin{itemize}
\item %We give conditions through which convergence can be deduced for the subgradient method (which is a \emph{non-smooth} system due to the presence {\color{black}of projections}) via the study of solutions to explicit \emph{smooth linear} ODEs derived from the form of the concave-convex function and the convex domain.
    {\color{black}We show that the subgradient dynamics, despite being nonlinear and non-smooth, have an $\omega$-limit set that is comprised of solutions to only {\em linear} ODEs.}

%\item These %smooth
%ODEs are {\color{black}shown to be the limiting solutions of the} subgradient method on \emph{affine subspaces}. {\color{black}This} links with part I \cite{Holding-Lestas-gradient-method-Part-I} of this two part work, where the convergence properties of these smooth systems are studied, {\color{black}thus allowing to deduce various conditions for the convergence of subgradient dynamics to a saddle point.}
\item
%These %smooth
%ODEs are {\color{black}
{\color{black}These solutions are
shown to belong to the $\omega-$limit set of the} subgradient method on \emph{affine subspaces}. {\color{black}This} links with part I \cite{Holding-Lestas-gradient-method-Part-I} of this two part work, where {\color{black}the limiting solutions of such systems have been exactly characterized.}
%the convergence properties of these {\color{black}systems} are studied.
{\color{black}Based on this characterization of the limiting solutions, a convergence result for subgradient dynamics  is also presented.}
%, {\color{black}thus allowing to deduce various conditions for the convergence of %subgradient dynamics to a saddle point.}
\item \icl{Various examples that illustrate the results in the paper are presented.}
    %applications of the results above are considered. In particular, we}
    %give a proof of the convergence of the subgradient method applied to any %strictly concave-convex function for an arbitrary convex domain. %Furthermore, {\color{black}
    %we
    \icl{Applications are also provided to} %{results to
    {modification methods in network optimization that provide convergence guarantees while maintaining a decentralized structure in the dynamics.
    %we provide example applications of our results on the subgradient method to various %methods of modifying a concave-convex function to give convergence. In particular, we %give
    An application to the problem of {\color{black}multi-path routing} is also discussed.}
\end{itemize}
%The study of the (sub)gradient method was originated by Arrow, Hurwicz and Uzawa \cite{arrow}, who took a direct approach and established convergence of the subgradient method with positivity constraints under the assumption of strict concave-convexity \cite{arrow}. More recently, Feijer and Paganini \cite{Paganini} attempted to use the invariance principle for hybrid automata \cite{Hybrid} to modernise and unify the ad-hoc approaches that had dominated until then. They provided a new proof of the convergence result of Arrow, Hurwicz and Uzawa, and also proved convergence of a number of modification methods where strict concave-convexity is absent. However, recently it has been pointed out by Cherukuri, Mallada and Cort\'es \cite{Cherukuri-primal-dual} that there are cases where, when interpreted as a hybrid automata, the subgradient method does not satisfy all the assumptions of the invariance principle in \cite{Hybrid}. In \cite{Cherukuri-primal-dual} an invariance principle for Carath\'eodory solutions is applied to prove convergence with positivity constraints for strictly concave-convex functions which are linear in the second variable (i.e. of the form \eqref{lagrangian-form}).

\section{Main Results}
\label{sec:Main}
This section states the main results of the paper.
%The aim of this work is to study the convergence properties of the subgradient method %(\autoref{def:subgradient-method}) applied to general concave-convex functions which lack %strict concavity and on an arbitrary convex domain.
%We divide the results into three parts, which are outlined below for the convenience of the %reader.
{\color{black}The results are divided into three subsections. To facilitate the readability of section \ref{sec:Main} we outline below the main Theorems that will be presented and the way these are related.

%\begin{itemize}
%\item
In \autoref{subsec:Main:faces} we consider pathwise stable semiflows,
an abstraction we use for the subgradient dynamcis in order to develop tools for their analysis that are valid despite their non-smooth character.
 %more abstract class of semiflows that includes the subgradient dynamics %\eqref{gradmethod-convex-domain}.
In particular,
\autoref{cor:pathwise-stable-plus-equilibrium-point-implies-isometries}
%in \autoref{subsec:Main:faces}
gives an invariance principle for such semi-flows, which applies without any smoothness assumption on the dynamics. We then additionally incorporate projections that constrain the trajectories within a closed convex set. Our key result, \autoref{prop:omega-limit-set-in-face}, says that for these semi-flows the {\em dynamics on the $\omega$-limit set are smooth}.

%In \autoref{subsec:Main:faces} we describe the essential problems that arise in analysis of the non-smooth dynamics of the subgradient method, and then develop tools to deal with this problem. In particular,
%\autoref{cor:pathwise-stable-plus-equilibrium-point-implies-isometries}
%%in \autoref{subsec:Main:faces}
%gives an {\em invariance principle for pathwise stable semi-flows}, which applies without any smoothness assumption on the dynamics. We then study semi-flows generated by a pathwise stable ODE with projections in the vector field that constrain the trajectories within a closed convex set. Our key result, \autoref{prop:omega-limit-set-in-face}, says that for these semi-flows the {\em dynamics on the $\omega$-limit set are smooth}.
%%\item

In \autoref{subsec:Main:subgradient-method} we apply these tools to the subgradient method \eqref{gradmethod-convex-domain}. In \autoref{thm:subgradient-method-faces} we show that the limiting solutions of the (non-smooth) subgradient method on a convex set are given by the dynamics of the (smooth) subgradient method on an {\em affine subspace}.
%, and describe exactly the set of limiting solutions when there is an internal saddle point.
This allows us to obtain \autoref{cor:subgradient-method-face-criterion}, a criterion for global asymptotic stability of the subgradient method.
%\item

In \autoref{subsec:Main:gradmethod} we combine \autoref{thm:subgradient-method-faces} with the results of Part I of this work \cite{Holding-Lestas-gradient-method-Part-I} (for convenience of the reader reproduced in {\autoref{app:projected-gradient-method}}) to obtain a general convergence criterion (\autoref{thm:subgrad-method-final-convergence-criterion}) for the subgradient method.
%\end{itemize}

These results are illustrated with examples throughout. The proofs of the results are given in appendix \ref{sec:proofs}.
}
%The results of these subsections are illustrated by examples of their application in \autoref{sec:modification-methods}, where they are used to prove convergence of the subgradient method in the case of some modified Lagrangian methods, which are then applied to the problem of multi-path congestion control.

\icl{
\subsection{Subgradient method on affine subspaces}\label{app:projected-gradient-method}
{In this section we recall} a result proved in part I of this work \cite{Holding-Lestas-gradient-method-Part-I} on the limiting solutions of the subgradient method on affine subspaces. %To prese this result we
To state this result we recall from \cite{Holding-Lestas-gradient-method-Part-I} the definition of the following matrices of partial derivatives of a concave-convex function $\varphi\in C^2$
% make use of the matrices
%$\mathbf{A}(\mathbf{z}), \mathbf{B}(\mathbf{z})$
%\eqref{def-of-A-and-B}}.

\begin{equation}\label{def-of-A-and-B}
\begin{aligned}
\mathbf{A}(\mathbf{z})&=\begin{bmatrix}
0&\varphi_{xy}(\mathbf{z})\\
-\varphi_{yx}(\mathbf{z})&0
\end{bmatrix}\\
\mathbf{B}(\mathbf{z})&=\begin{bmatrix}
\varphi_{xx}(\mathbf{z})&0\\
0&-\varphi_{yy}(\mathbf{z})
\end{bmatrix}.
\end{aligned}
\end{equation}
%recall from \cite{Holding-Lestas-gradient-method-Part-I} the definition of the following matrices of partial derivatives of $\varphi$.
%\begin{equation}\label{def-of-A-and-B}
%\begin{aligned}
%\mathbf{A}(\mathbf{z})&=\begin{bmatrix}
%0&\varphi_{xy}(\mathbf{z})\\
%-\varphi_{yx}(\mathbf{z})&0
%\end{bmatrix}\\
%\mathbf{B}(\mathbf{z})&=\begin{bmatrix}
%\varphi_{xx}(\mathbf{z})&0\\
%0&-\varphi_{yy}(\mathbf{z})
%\end{bmatrix}.
%\end{aligned}
%\end{equation}
%Consider the ODE \eqref{eq:subgradient-method-limiting-solutions} in more detail.
{Consider the subgradient method \eqref{gradmethod-convex-domain} on an affine subspace $V$ with normal cone $N_V$
\begin{align}\label{eq:affine_proj}
\dot{\mathbf{z}}&=\mathbf{f}(\mathbf{z})-\mathbf{P}_{N_V}(\mathbf{f}(\mathbf{z}))\\
\mathbf{f}(\mathbf{z})&=\icl{\begin{bmatrix}
                        \varphi_x\\ %&\!\!\!
			-\varphi_y
                       \end{bmatrix}}.\nonumber
\end{align}
%\eqref{eq:subgradient-method-limiting-solutions}
%where $V$ is an affine subspace and $N_V$ its normal cone.
Also let} $\mathbf{\Pi}\in\mathbb{R}^{(n+m)^2}$ be the orthogonal projection matrix onto the orthogonal complement of $N_V$. Then the ODE \eqref{eq:affine_proj} can be written as
\begin{equation}\label{eq:projected-gradient-method}
\dot{\mathbf{z}}%=\mathbf{f}(\mathbf{z})-\mathbf{P}_{N_V}(\mathbf{f}(\mathbf{z}))
=\mathbf{\Pi}\mathbf{f}(\mathbf{z})
\end{equation}
%where $\mathbf{f}(\mathbf{z})=\begin{bmatrix}
%                        \varphi_x&\!\!\!
%			-\varphi_y
%                       \end{bmatrix}^T$.
The result is stated for $\mathbf{0}$ being an equilibrium point; the general case may be obtained by a translation of coordinates.
\begin{theorem}\label{thm:projected-gradient-method-result}
%\cite[{\autoref{I-thm:projected-gradient-method-result}}]{Holding-Lestas-gradient-method-Part-I}
\cite[Theorem 25]{Holding-Lestas-gradient-method-Part-I}
Let $\mathbf{\Pi}\in\mathbb{R}^{(n+m)^2}$ be an orthogonal projection matrix, $\varphi$ be $C^2$ and concave-convex on $\mathbb{R}^{n+m}$, and $\mathbf{0}$ be an equilibrium point of \eqref{eq:projected-gradient-method}. Then the trajectories $\mathbf{z}(t)$ of \eqref{eq:projected-gradient-method} that lie a constant distance from any equilibrium point of \eqref{eq:projected-gradient-method} are exactly the solutions to the linear ODE:
\begin{equation}\label{eq:app1}
\dot{\mathbf{z}}(t)=\mathbf{\Pi}\mathbf{A}(\mathbf{0})\mathbf{\Pi}\mathbf{z}(t)
\end{equation}
that satisfy, for all $t\in\mathbb{R}$ and $r\in[0,1]$, the condition
\begin{equation}\label{eq:app2}
\!\!\mathbf{z}(t)\in\ker(\mathbf{\Pi}\mathbf{B}(r\mathbf{z}(t))\mathbf{\Pi})\cap \ker(\mathbf{\Pi}(\mathbf{A}(r\mathbf{z}(t))-\mathbf{A}(\mathbf{0}))\mathbf{\Pi})
\end{equation}
where $\mathbf{A}(\mathbf{z})$ and $\mathbf{B}(\mathbf{z})$ are defined by \eqref{def-of-A-and-B}.
\end{theorem}
\begin{remark}
%As it will be discussed in the paper subgradient dynamics are known to be pathwise stable. So a direct application of Lasalle's theorem to system \eqref{eq:projected-gradient-method} (which is smooth)  leads to the fact that all solutions of this system converge to the set of trajectories described in \autoref{thm:projected-gradient-method-result}. In particular,
%%trajectories of converge to the set of solutions that lie a constant distance from all %equilibrium points. \autoref{thm:projected-gradient-method-result} shows that
%despite the nonlinearity of the dynamics this set is comprised of only explicit linear ODEs.
In the remainder of this paper we show that subgradient dynamics on a general convex domain that have an equilibrium point, have an $\omega$-limit set that is comprised of solutions of subgradient dynamics on an only an affine subspace and form a flow of isometries. In particular, the $\omega$-limit set is comprised of solutions to explicit linear ODEs, of the form described in \autoref{thm:projected-gradient-method-result}, despite the subgradient dynamics being nonlinear and non-smooth.
\end{remark}
}
\subsection{Pathwise stability and convex projections}\label{subsec:Main:faces}
If one wishes to extend the results of Part I of this work \cite{Holding-Lestas-gradient-method-Part-I} to the subgradient method on a non-empty closed convex set $K\subseteq\mathbb{R}^{n+m}$, then one runs into two problems, both coming from the discontinuity of the vector field in \eqref{gradmethod-convex-domain}. The first is that the previously simple application of LaSalle's theorem would become much more technical - needing tools from non-smooth analysis. The second, more fundamental, problem is that LaSalle's theorem only gives convergence to a set of trajectories, and it remains to characterise this set. The trajectories in this set still satisfy an ODE with a discontinuous vector field, and we do not have uniqueness of the solution backwards in time - we {\color{black}still, though, have} a semi-flow.

To solve these issues we reinterpret the prior results in terms of a simple property which is still present in the subgradient method.

The main tool used to prove the results in \cite{Holding-Lestas-gradient-method-Part-I} was pathwise stability, (\autoref{def:pathwise-stability}), which says that the Euclidean distance between any two solutions is non-increasing with {\color{black}time (we will} later prove such a result for the subgradient method). {\color{black}Intuitively, one would expect that the distance between any two of the limiting solutions would be constant.
%, and indeed, one can verify that this is the case directly from %\eqref{eq:projected-gradient-method-linear-ODE} (below) and the skew-symmetry of %$\mathbf{A}(\mathbf{0})$ given by \eqref{def-of-A-and-B}.
A more abstract way of saying this is that the sub-flow obtained by considering the gradient method with initial conditions in the $\omega$-limit set
%acting on the set of limiting solutions
is a \emph{flow of isometries}. In fact, this can be proved %more directly
for any pathwise stable semi-flow, as stated in Proposition \ref{cor:pathwise-stable-plus-equilibrium-point-implies-isometries} below \icl{(proved in Appendix \ref{sec:isometry}).}}
\begin{proposition}\label{cor:pathwise-stable-plus-equilibrium-point-implies-isometries}
Let $(\phi,X,d)$ be a pathwise stable semi-flow (see \autoref{def:pathwise-stability}) with $X\subseteq\mathbb{R}^{n+m}$ which has an equilibrium point $\mathbf{\bar{z}}$. Let $\Omega$ be its $\omega$-limit set. Then the sub-semi-flow $(\phi,\Omega,d)$ (see \autoref{def:sub-flow}) defines a flow of isometries (see \autoref{def:flow-of-isometries}). Moreover, $\Omega$ is a convex set.
\end{proposition}
Note here that $(\phi,\Omega,d)$ is a \emph{flow} rather than a \emph{semi-flow}. This comes from the simple observation that an isometry is always invertible, so we can define, for $t\ge0$, $\phi(-t,\cdot):\Omega\to\Omega$ as $\phi(t,\cdot)^{-1}$.
\begin{remark}\label{rem:backwards-flow-interpretation}
Care should be taken in interpreting the backwards flow given by \autoref{cor:pathwise-stable-plus-equilibrium-point-implies-isometries}. There could be multiple trajectories in $X$ that meet at a point in $y\in \Omega$ at time $t=0$, but exactly one of these trajectories will lie in $\Omega$ for all times $t\in\mathbb{R}$.
\end{remark}
We would like to note that we are not the first to make this observation. Indeed, we deduce this result from a more general result in \cite{Giacomo-equicontinuity} which was published in 1970.

\icl{It should be noted that if a pathwise stable semiflow %in %\eqref{eq:projected-pathwise-stable-ODE}
has an equilibrium point then all its trajectories are bounded. % from. %\autoref{Incremental-stability-is-preserved-by-convex-projection}.
This implies that each trajectory converges to its set of $\omega$-limit points \cite[Lemma 4.16]{meiss2007differential}.
\icl{The structure of the $\omega$-limit set can also be used to strengthen the convergence to the $\omega$-limit set to convergence to a solution in the $\omega$-limit set. This is stated as
\autoref{prop:converge_Omega} below (proved in Appendix \ref{sec:isometry}).
%states that the convergence to the $\omega$-limit set $\Omega$ can be strengthened to a convergence %to a solution in $\Omega$. % (defined in the text below \autoref{def:semiflow}).
%Furthermore if a solution has an $\omega$-limit point that is an equilibrium point then %\autoref{Incremental-stability-is-preserved-by-convex-projection} implies convergence of the solution to this point. %Therefore if the $\omega$-limit set consists only of equilibrium points then we have global convergence.
%In \autoref{} in  we strengthen the convergence to the $\omega$-limit set $\Omega$ to convergence to a solution in $\Omega$.
%Note that if the $\omega$-limit set consists only of equilibrium points then global convergence can be deduced. This is due to the boundedness of solutions which implies that every solution has an $\omega$-limit point. If an  $\omega$-limit point of a solution is an equilibrium point, then convergence of the solution to this point can be easily deduced from the pathwise stability property.
%\end{remark}
\begin{corollary}\label{prop:converge_Omega}
Let $(\phi,X,d)$ be a pathwise stable semi-flow with $X\subseteq\mathbb{R}^{n+m}$, which has an equilibrium point.
%Let \eqref{eq:projected-pathwise-stable-ODE} hold and assume that the semi-flow $(\phi,K,d)$ %has an equilibrium point. %Let %$\Omega$ be its $\omega$-limit set.
Then each trajectory of the semiflow %$(\phi,K,d)$
converges to a trajectory in its $\omega$-limit set.
\end{corollary}}}

We consider \icl{now} pathwise stable differential equations which are projected onto a convex set, and make the following set of {\color{black}assumptions.}
%\footnote{That the final assumed inequality in \eqref{eq:projected-pathwise-stable-ODE} holds %for the subgradient method is evident from the proof of the pathwise stability of the gradient %method presented in \cite[Appendix B]{Holding-Lestas-gradient-method-Part-I}.}
 \begin{equation}\label{eq:projected-pathwise-stable-ODE}
 \begin{aligned}
 (\phi,K,d)\text{ is }&\text{the semi-flow of Carath\'eodory solutions of}\\
 \dot{\mathbf{z}}&=\mathbf{f}(\mathbf{z})-\mathbf{P}_{N_K(\mathbf{z})}(\mathbf{f}(\mathbf{z}))\text{ where,}\\
 K&\subseteq\mathbb{R}^{\icl{n+m}},\text{ is non-empty, closed and convex}\\
 C^1\ni\mathbf{f}:&K\to\mathbb{R}^{\icl{n+m}}\text{ satisfies, for all $\mathbf{z},\mathbf{w}\in K$,}\\
(\mathbf{f}&(\mathbf{z})-\mathbf{f}(\mathbf{w}))^T(\mathbf{z}-\mathbf{w})\le0.
 \end{aligned}
 \end{equation}
{\icl{%It should be noted that
Functions $-\mathbf{f}(\mathbf{z})$  such that $\mathbf{f}(\mathbf{z})$ satisfies the final inequality in \eqref{eq:projected-pathwise-stable-ODE} are referred to as monotone.
%It should be noted that this inequality is satisfied by $\mathbf{f}(\mathbf{z})$ in \eqref{gradmethod-convex-domain}. It should also be noted that this inequality is also satisfied if $\mathbf{f}(\mathbf{z})$ is replaced by $\mathbf{f}(\mathbf{z})-\mathbf{P}_{N_K(\mathbf{z})}(\mathbf{f}(\mathbf{z}))$, from the definition of the normal cone.
A known result in the literature %that follows from corresponding properties of monotone operators,
%that follows from this property of $\mathbf{f}(\mathbf{z})$
is the fact that the semi-flow in \eqref{eq:projected-pathwise-stable-ODE} is pathwise stable\footnote{\icl{In particular note that for any two trajectories $\mathbf{z}, \mathbf{z}'$
%, and $\delta \mathbf{z}:=\mathbf{z}(t)-\mathbf{z}'(t)$
we have that $W(t)=\frac12|\mathbf{z}(t)-\mathbf{z}'(t)|^2$, satisfies for almost all times $t\ge0$ ,
% \begin{align*}
$
\dot{W}(t)=(\mathbf{z}(t)-\mathbf{z}'(t))^T(\mathbf{\dot z}(t)-\mathbf{\dot z}'(t))\leq0
$
% \\
% &=(\mathbf{z}(t)-\mathbf{z}'(t))^T(\mathbf{z}(t)-\mathbf{z}'(t))+\\
% &\quad %-(\mathbf{z}(t)-\mathbf{z}'(t))^T\mathbf{P}_{N_K(\mathbf{z}(t))}(\mathbf{f}(\mathbf{z}(t)))+
% %(\mathbf{z}(t)-\mathbf{z}'(t))^T\mathbf{P}_{N_K(\mathbf{z}'(t))}(\mathbf{f}(\mathbf{z}'(t)))\\
% &\leq0
% \end{align*}
%This inequality %is implied by the monotonicity for the right-hand side of the ODE in \eqref{eq:projected-pathwise-stable-ODE} and
where the inequality follows from \eqref{eq:projected-pathwise-stable-ODE} and the definition of the normal cone.
%, i.e. $\dot{W}(t) %=(\mathbf{z}(t)-\mathbf{z}'(t))^T(\mathbf{f}(\mathbf{z}(t))-\mathbf{f}(\mathbf{z}'(t)))
%-(\mathbf{z}(t)-\mathbf{z}'(t))^T\mathbf{P}_{N_K(\mathbf{z}(t))}(\mathbf{f}(\mathbf{z}(t)))
%+(\mathbf{z}(t)-\mathbf{z}'(t))^T\mathbf{P}_{N_K(\mathbf{z}'(t))}(\mathbf{f}(\mathbf{z}'(t))).$}
% The first term is non-positive due to the inequality in %\eqref{eq:projected-pathwise-stable-ODE}.
% and the other two terms due to the definition of the normal cone.
}}, which is stated as Lemma \ref{Incremental-stability-is-preserved-by-convex-projection} below
%This follows from the fact that the negative of the right-hand side of the ODE in %\eqref{eq:projected-pathwise-stable-ODE} is monotone
\cite{goebel2017stability}, \cite{rockafellar1970convex}.
The monotonicity property of $-\mathbf{f}(\mathbf{z})$
%and appropriate extensions of the right-hand side to a maximal monotone operator
allows also to deduce existence of unique solutions $\mathbf{z}:[0,\infty)\rightarrow \mathbb{R}^{n+m}$ \cite[Theorem~1]{brogliato2006equivalence} (see also \cite{goebel2017stability}, \cite{venets1985continuous}). }}
%holds for the subgradient method \eqref{gradmethod-convex-domain}, which is evident from %the proof of the pathwise stability of the gradient method presented in \cite[Appendix %B]{Holding-Lestas-gradient-method-Part-I}.}

% A simple first result is that the projected dynamics are still pathwise stable.
 \begin{lemma}\label{Incremental-stability-is-preserved-by-convex-projection}
 Let \eqref{eq:projected-pathwise-stable-ODE} hold. Then $(\phi,K,d)$ is pathwise stable.
 \end{lemma}
 %\icl{
%\begin{remark}

%Furthermore, it follows from \autoref{Incremental-stability-is-preserved-by-convex-projection} %one can directly deduce
%that if a solution has an $\omega$-limit point that is an equilibrium point, then it also converges to this point~\cite{goebel2017stability}.}

%This hence leads to the corollary below.
%%The corollary below follows also %from \autoref{prop:converge_Omega} but can be also deduced
%%directly from \autoref{Incremental-stability-is-preserved-by-convex-projection}.
%\begin{corollary}
%Let \eqref{eq:projected-pathwise-stable-ODE} hold. If the $\omega$-limit set of  semi-flow $(\phi,K,d)$ includes only equilibrium points then the semi-flow is globally convergent.
%\end{corollary}
%}

Our main result on projected differential equations \icl{as in \eqref{eq:projected-pathwise-stable-ODE}} is that, even though the projection term gives a discontinuous vector field, when we restrict our attention to the $\omega$-limit set, the vector field is $C^1$. This allows us to replace \emph{non-smooth analysis} with \emph{smooth analysis} when studying the asymptotic behaviour of such systems.
\begin{theorem}\label{prop:omega-limit-set-in-face}
Let \eqref{eq:projected-pathwise-stable-ODE} hold and assume that the semi-flow $(\phi,K,d)$ has an equilibrium point. Let $\Omega$ be its $\omega$-limit set. Then $(\phi,\Omega,d)$ defines a flow of isometries given by solutions to the following differential equation, which has a $C^1$ vector field,
\begin{equation}\label{eq:projected-ODE-on-face}
\mathbf{\dot{z}}=\mathbf{f}(\mathbf{z})-\mathbf{P}_{N_V}(\mathbf{f}(\mathbf{z})).
\end{equation}
Here $V$ is the affine span of the (unique) minimal face (see \autoref{def:minimal-face}) of $K$ that contains the set of equilibrium points of the {\color{black}semi-flow.}
%If $\mathbf{f}$ is defined on all of $V$ then $\Omega$ is contained in the $\omega$-limit set %of the flow generated by \eqref{eq:projected-ODE-on-face} on $V$.
\end{theorem}
\icl{The proof of \autoref{prop:omega-limit-set-in-face} is provided in Appendix \ref{sec:isometry}.}
\begin{remark}
The existence of a minimal face of $K$ that contains the set of equilibrium points is a simple consequence of the definition of a face (see \autoref{def:face} and the discussion that follows). \icl{The significance of minimal face flows was also noted in \cite{zhang1995stability}, where these have been used as a tool to deduce local stability properties for projected dynamical systems. In  \autoref{prop:omega-limit-set-in-face} we show that such flows can provide a characterization to the $\omega$-limit set of dynamical systems that are pathwise stable.}
%The important part of \autoref{prop:omega-limit-set-in-face} is that the dynamics on $\Omega$ %are given by \eqref{eq:projected-ODE-on-face}, i.e. the projection operator %$\mathbf{P}_{N_K(\mathbf{z})}$ in \eqref{gradmethod-convex-domain} becomes $\mathbf{P}_{N_V}$ %which does not depend on the {\color{black}position $\mathbf{z}$.
\icl{Noting also that \eqref{eq:projected-ODE-on-face} is a dynamical system on an affine subspace,  \autoref{thm:projected-gradient-method-result} can be used to provide a characterization to the $\omega$-limit set of subgradient dynamics as linear ODEs, as it will be discussed in the next section.}
\end{remark}
%\begin{remark}
%In the statement of \autoref{prop:omega-limit-set-in-face}, it is not necessarily the case that
% %  \begin{equation}
%$\mathbf{P}_{N_V}(\mathbf{f}(\mathbf{z}))=\mathbf{P}_{N_K(\mathbf{z})}(\mathbf{f}(\mathbf{z}))$
% %  \end{equation}
%for all $\mathbf{z}$ in the minimal face provided by the {\color{black}Theorem; this} is only guaranteed for $\mathbf{z}\in\Omega$. Neither is it the case that $N_K(\mathbf{z})$ must be constant and equal to $N_V$ {\color{black}on $\Omega$; it only holds} that the projection of $\mathbf{f}(\mathbf{z})$ onto $N_V$ and $N_K(\mathbf{z})$ must be equal.
%\end{remark}
%The minimal face $F$ in \autoref{prop:omega-limit-set-in-face} is not given explicitly, so that $N_V$ cannot be computed. Instead, we must study the dynamics \eqref{eq:projected-ODE-on-face} with only the knowledge that $V$ is the affine span of a face of $K$. For specific sets $K$ this can give a lot of information, see e.g. \Autoref{example:faces-of-a-strictly-convex-set,example:faces-corresponding-to-positivity-constraints} below, but in general $V$ could be any affine subspace.

}
\subsection{\color{black} The subgradient method}\label{subsec:Main:subgradient-method}
We now apply \icl{\autoref{prop:omega-limit-set-in-face}} to the subgradient method. Our first result reduces the study of the convergence on general convex domains, where the subgradient method is non-smooth, to the study of convergence of the subgradient method on affine spaces, {\color{black}which is a smooth dynamical system studied in \cite{Holding-Lestas-gradient-method-Part-I}. We also show that when an internal saddle point exists then the limiting behaviour of the subgradient method is determined by that of the corresponding unconstrained gradient method.}
%We also give an exact classification of the asymptotic behaviour in the case of an internal %saddle point.

As in part I of this work \cite{Holding-Lestas-gradient-method-Part-I}, given a concave-convex function $\varphi$ we define {\color{black} the following
\begin{itemize}
\item $\bar{\mathcal{S}}$ is the set of saddle points of $\varphi$
 \item $\mathcal{S}$ %respectively as the set of saddle {\color{black}points, and}
is the set of solutions to the {\em gradient method \eqref{gradmethod-fullspace}} (i.e. no projections included) that lie a constant distance from any saddle point.
\end{itemize}
%Note that the set $\mathcal{S}$ has been exactly characterized in part~I~\cite{Holding-Lestas-gradient-method-Part-I}.
}
\begin{theorem}\label{thm:subgradient-method-faces}
%Let $K\subseteq\mathbb{R}^{n+m}$ be non-empty, closed and convex. Let $\varphi$ be $C^2$, concave-convex on $K$ and have a $K$-restricted saddle point.
\icl{Let function $\varphi$ be $C^2$, and concave-convex on a set $K\subseteq\mathbb{R}^{n+m}$ as defined in \eqref{eq:K_phi}, and let $\varphi$ have a $K$-restricted saddle point.}
Let $(\phi,K,d)$ denote the subgradient method \eqref{gradmethod-convex-domain} on $K$ and $\Omega$ be its $\omega$-limit set. %Then {\color{black}$(\phi,\Omega,d)$ defines a flow of isometries and the following hold.}
Then $\Omega$ is convex, and $(\phi,\Omega,d)$ defines a flow of isometries. {\color{black}Furthermore, the following hold:}

%Exactly one of the following holds:
\begin{enumerate}[(i)]

\item %{\color{black}There is no saddle point of $\phi$ in the interior of $K$}.
%, i.e. $\bar{\mathcal{S}}\cap\interior K=\emptyset$}.
 %, which is proper.
The trajectories $\mathbf{z}(t)$ of $(\phi,\Omega,d)$ solve the ODE:
\begin{equation}\label{eq:subgradient-method-limiting-solutions}
\dot{\mathbf{z}}=\mathbf{f}(\mathbf{z})-\mathbf{P}_{N_V}(\mathbf{f}(\mathbf{z})),
\end{equation}
{\color{black} where $V$ is the affine span of $F$, with $F$ being the minimal face containing all $K$-restricted saddle points.}
%where $\mathbf{f}(\mathbf{z})=\begin{bmatrix}
%                        \varphi_x&\!\!\!
%			-\varphi_y
%                       \end{bmatrix}^T$.
                       %Furthermore, if $\varphi$ is also concave-convex on $V$ then $\Omega$ is contained in the $\omega$-limit set of the subgradient method \eqref{gradmethod-convex-domain} on $V$.

\item
%There is an internal saddle point $\mathbf{\bar{z}}\in \bar{\mathcal{S}}\cap \interior K$ %and
{\color{black} {\color{black}If there exists a saddle point of $\varphi$ in the interior of $K$, then}}
%i.e. $\bar{\mathcal{S}}\cap \interior K\neq\emptyset$, and}
 % $$\mathbf{\bar{z}}\in \bar{\mathcal{S}}\cap \interior K$ %and
\begin{equation}\label{eq:omega-is-S-cap-K}
\Omega=\{\mathbf{z}(t)\in\mathcal{S}:\mathbf{z}(\mathbb{R})\subseteq K\}.
\end{equation}
{\color{black}where $\mathcal{S}$ is as defined before the theorem statement.}
\end{enumerate}
\end{theorem}
\icl{The proof of \autoref{thm:subgradient-method-faces} is provided in Appendix \ref{sec:subgradient-method}.}

{\color{black}
\begin{remark}
The ODE \eqref{eq:subgradient-method-limiting-solutions} is the subgradient method on the affine subspace $V$. {\color{black}A main significance of \autoref{thm:subgradient-method-faces} is the fact that {\color{black}the solutions of \eqref{eq:subgradient-method-limiting-solutions}} in $\Omega$ can be characterized using the results in part I \cite{Holding-Lestas-gradient-method-Part-I}. In particular, it follows
%\footnote{\color{black}More precisely, using the equivalent representation \eqref{eq:projected-gradient-method} of \eqref{eq:subgradient-method-limiting-solutions} and the isometry property of $\Omega$, Theorem \ref{thm:projected-gradient-method-result} can be used to deduce that all trajectories $\mathbf{z}(t)$ in $\Omega$ satisfy \eqref{eq:app1} and \eqref{eq:app2}.}
%% (see also \autoref{rem:OmLinear}).}
from Theorem \ref{thm:projected-gradient-method-result} in {section \ref{app:projected-gradient-method}} that these
satisfy explicit linear ODEs.
%the limiting solutions of \eqref{eq:subgradient-method-limiting-solutions}} were classified in \cite{Holding-Lestas-gradient-method-Part-I}. {\color{black}In particular, it was shown that they satisfy explicit linear ODEs, of the form given in Theorem \ref{thm:projected-gradient-method-result} in the appendix.
This therefore shows that even though the subgradient dynamics are nonlinear and nonsmooth their $\omega$-limit set is comprised of solutions to only {\em linear} ODEs (stated in  \autoref{cor:subgradient_linear}).}
\end{remark}

\begin{remark}
{\color{black}Later}, in \autoref{subsec:Main:gradmethod} we use
the results in \cite{Holding-Lestas-gradient-method-Part-I} {\color{black}on the subgradient} method on affine subspaces \icl{(section \ref{app:projected-gradient-method})}
together with \autoref{thm:subgradient-method-faces} to obtain a convergence criterion for the subgradient method. This is used subsequently to give proofs for the applications considered in \autoref{sec:modification-methods}.
\end{remark}

\begin{remark}
{\color{black}
  It will be discussed in the proof  of \autoref{thm:subgradient-method-faces} that \autoref{thm:subgradient-method-faces}(ii) is a special case
  %\footnote{\color{black}This is due to the fact that in (i) the (non-proper) face {\color{black}$F$} defined in (ii) is the whole set $K$, and the affine span of $K$ is $\mathbb{R}^{n+m}$ {\color{black}}, thus making the projection term in \eqref{eq:subgradient-method-limiting-solutions} equal to zero.}
  of \autoref{thm:subgradient-method-faces}(i) where the projection term in \eqref{eq:subgradient-method-limiting-solutions} equal to zero.}
%Note that \autoref{thm:subgradient-method-faces}(i) is a special case of \autoref{thm:subgradient-method-faces}(ii). {\color{black} In particular, in (i) the (non-proper) face {\color{black}$F$} defined in (ii) is the whole set $K$, and the affine span of $K$ is $\mathbb{R}^{n+m}$ {\color{black}}, thus making the projection term in \eqref{eq:subgradient-method-limiting-solutions} equal to zero.}
%\end{remark}
%\begin{remark}
In \autoref{thm:subgradient-method-faces}(ii) %is stated separately, since in this case
there is a simple characterization of the \icl{$\omega$-limit set} of the subgradient method, as just the limiting solutions of the corresponding {\em gradient method} that lie in $K$.  Note  that the set $\mathcal{S}$ in \eqref{eq:omega-is-S-cap-K} was exactly {\color{black}characterized in \cite{Holding-Lestas-gradient-method-Part-I}.}
%, thus providing  a {\color{black}simple} full characterisation of the limiting solutions of the subgradient method when there is an internal saddle point. In applications, however, it is common for the latter property to not hold, e.g. in the case of a Lagrangian originating from an optimisation problem where at least one of the inequality constraints is binding. In such cases \autoref{thm:subgradient-method-faces}(ii) {\color{black}applies.} % and gives a smooth ODE that the limiting solutions must solve.
\end{remark}
\begin{remark}
 A simple consequence of \eqref{eq:omega-is-S-cap-K} is the fact if there exists a saddle point in the interior of $K$ then the subgradient method is globally convergent if the corresponding unconstrained gradient method is globally convergent.
 \end{remark}}
\icl{
\begin{remark}
 \autoref{thm:subgradient-method-faces} follows directly from \autoref{prop:omega-limit-set-in-face}. Hence \autoref{thm:subgradient-method-faces} also holds when $K$ is an arbitrary closed convex set as in  \autoref{prop:omega-limit-set-in-face} (rather than just the cartesian product of two convex sets),
%as in \eqref{eq:K_phi},
if the subgradient method has an equilibrium point. Note that  a saddle point in the interior of $K$,  %(assumed to exist in \autoref{thm:subgradient-method-faces}(ii)),
or a $K$-restricted saddle point with $K$ as in \eqref{eq:K_phi}, is always an equilibrium point of the subgradient method on~$K$.
%It should be noted that \autoref{thm:subgradient-method-faces}(ii) also holds when $K$ is an arbitrary closed convex set, rather than just the cartesian product of two convex sets as in \eqref{eq:K_phi}. This follows easily from the proof of the Theorem and the fact that \autoref{thm:subgradient-method-faces} follows directly from \autoref{prop:omega-limit-set-in-face}. In particular,
% %It easily follows from the proof of  This is due to the fact that
% a saddle point in the interior of $K$ is always an equilibrium point of the subgradient method on $K$, as required in  \autoref{prop:omega-limit-set-in-face}, and also in \autoref{prop:omega-limit-set-in-face}
% %. Also in \autoref{prop:omega-limit-set-in-face}
% the set $K$ is not restricted to be the cartesian product of two convex sets. % in this Thereom.
\end{remark}
}
%The corresponding result for the subgradient method with constant gains can be obtained by a coordinate transformation as described in \cite[Appendix]{Holding-Lestas-gradient-method-Part-I}.

%
%\begin{remark}
%\autoref{thm:subgradient-method-faces}(i) and the results in \cite{Holding-Lestas-gradient-method-Part-I} give a full characterisation of the limiting solutions of the subgradient method when there is an internal saddle point, however, in applications it is common for this to not hold, e.g. in the case of a Lagrangian originating from an optimisation problem where at least one of the inequality constraints is binding. In such cases \autoref{thm:subgradient-method-faces}(ii) applies and gives a smooth ODE that the limiting solutions must solve.
%\end{remark}

We now present several examples to illustrate the application of \autoref{thm:subgradient-method-faces} in some simple cases.
%\begin{example}\label{example:faces-of-a-strictly-convex-set}
%Consider the case where $K\subseteq\mathbb{R}^{n+m}$ is strictly convex. In this case the proper faces of $K$ are given by $\{\mathbf{w}\}$ for each $\mathbf{w}\in\partial K$, i.e. each consist of a single point of the boundary of $K$. The subgradient method on a single point is trivially globally
%%\footnote{Note that by global convergence of the subgradient method we refer ro convergence for %all initial conditions in the domain the subgradient method is providing a restriction. As this %region here is a just point then the convergence is trivial.}
%convergent.
%
%By applying the two cases of \autoref{thm:subgradient-method-faces}, we obtain that:
%\begin{enumerate}[(i)]
%\item If there is a saddle point in the interior of $K$, then the subgradient method \eqref{gradmethod-convex-domain} on $K$ is globally {\color{black}convergent if the} (unconstrained) gradient method \eqref{gradmethod-fullspace} is globally convergent.
%\item If there is no saddle point in the interior of $K$, but there is a $K$-restricted saddle point, then the subgradient method is globally convergent.
%    %\footnote{Note that the affine span of a point is the point itself point}.
%\end{enumerate}
%\end{example}
%There are cases where the unconstrained gradient method \eqref{gradmethod-fullspace} is globally convergent, but the subgradient method is not, as the next example illustrates.

The first example corresponds to a case where the unconstrained gradient method \eqref{gradmethod-fullspace} is globally convergent, but the subgradient method is not.
\begin{example}\label{example:projections-break-convergence}
Define the concave-convex function
\begin{equation}
\varphi(x_1,x_2,y)=-\frac12|x_1|^2 + (x_1+x_2)y
\end{equation}
{\color{black}where $x_1,x_2,y\in\mathbb{R}$}.
This has a single saddle point at $(0,0,0)$, and {\color{black}$\varphi$ is the Lagrangian of} the optimisation problem
\begin{equation}
\max_{x_1+x_2=0}-\frac12|x_1|^2
\end{equation}
where {\color{black}variable $y$ in function $\varphi$ is the Lagrange multiplier associated with the constraint.} % the constraint is relaxed {\color{black}in $\varphi$} with the Lagrange multiplier $y$.
 On this function the gradient method is the linear system
\begin{equation}
\begin{bmatrix}
\dot{x}_1\\\dot{x}_2\\\dot{y}
\end{bmatrix}=\begin{bmatrix}
-1 & 0&1\\
0 & 0&1\\
-1 &-1 &0
\end{bmatrix}
\begin{bmatrix}
x_1\\x_2\\y
\end{bmatrix}.
\end{equation}
It is easily verified that all the eigenvalues of this matrix lie in the left half plane, so that the gradient method is globally convergent. Now consider the family of convex sets defined by
\begin{equation}\label{eq:K_a}
K_a=\{(x_1,x_2,y)\in\mathbb{R}^3:x_1\ge a\}
\end{equation}
for $a\in\mathbb{R}$. The subgradient method on $K_a$ is given by the system
\begin{equation}\label{eq:subgradient-method-on-F_a1}
\begin{aligned}
\dot{x}_1&=[-x_1+y]_{x_1-a}^+\\
\dot{x}_2&=y\\
\dot{y}&=-x_1-x_2.
\end{aligned}
\end{equation}
The convergence of the subgradient method on $K_a$ depends crucially on the value of $a$. There are three cases:
\begin{enumerate}[(i)]
\item $a<0$: In this case the saddle point $(0,0,0)$ lies in the interior of $K_a$ so that \autoref{thm:subgradient-method-faces}(ii) applies, and as the unconstrained gradient method is globally convergent, so is the subgradient method on $K_a$.
\item $a>0$: Here the unconstrained saddle point $(0,0,0)$ lies outside $K_a$. A simple computation shows that the point {\color{black}$(a,-a,0)$} is the only $K_a$-restricted saddle point. {\color{black}\autoref{thm:subgradient-method-faces}(i) can be used here}. The only proper face of $K_a$ is the set
\begin{equation}\label{eq:F_a}
F_a=\{(a,x_2,y):x_2,y\in\mathbb{R}\}.
\end{equation}
The subgradient method on $F_a$ is the system
\begin{equation}\label{eq:KaODE}
\begin{bmatrix}
\dot{x_2}\\\dot{y}
\end{bmatrix}
=
\begin{bmatrix}
0&1\\
-1&0\\
\end{bmatrix}
\begin{bmatrix}
x_2\\y
\end{bmatrix}
+\begin{bmatrix}
0\\-a
\end{bmatrix}
\end{equation}
together with the equality $x_1=a$. This matrix has imaginary eigenvalues $\pm i$, showing that the subgradient method on $F_a$ is not globally convergent.
{\color{black}It is easy to verify that some of these oscillatory solutions are also solutions of the subgradient method on $K_a$, \icl{e.g. $y(t)=a\cos(t)$, $x_2(t)=-a(1-\sin(t))$, $x_1(t)=a$ satisfy %both \eqref{eq:KaODE} and
\eqref{eq:subgradient-method-on-F_a1}}. Therefore the subgradient method on $K_a$ is not globally convergent when $a>0$.}

%This {\color{black}does not yet} imply that the subgradient method on $K_a$ is not globally convergent, as we have not verified that some subset of these oscillatory solutions to the subgradient method on $F_a$ are also solutions of the subgradient method on $K_a$. However, it is easy to verify that this is indeed the case, so that the subgradient method on $K_a$ is not globally convergent when $a>0$.
\item $a=0$: In this case the saddle point $(0,0,0)$ lies on the boundary of $K_0$. {\color{black}\autoref{thm:subgradient-method-faces}(i) applies}, and the analysis of the subgradient method on $F_0$ is the same as in case (ii) above. However, when we check whether any oscillatory solutions of the subgradient method on $F_0$ are also solutions of the subgradient method on $K_0$, we find that there are no such solutions. Indeed, {\color{black}for a trajectory to be} a solution {\color{black}to both the} subgradient method on $F_0$ and the subgradient method on $K_0$ we must have both $x_1=a=0$ and $-x_1+y\le 0$ by \eqref{eq:subgradient-method-on-F_a1}. Then \eqref{eq:subgradient-method-on-F_a1} implies that $y=0$ and then that $x_1=0$. So the only such solution is the saddle point. Therefore the subgradient method {\color{black}on $K_0$ is} globally convergent.
\end{enumerate}
This shows that the subgradient method on $K_a$ undergoes a bifurcation at $a=0$.
\end{example}
The following example illustrates that the subgradient method can be globally convergent when the gradient method is not.
\begin{example}\label{example:disappearance-of-saddle-points}
Define the concave-convex function
\begin{equation}
\varphi(x_1,x_2,y)=-\frac{1}{2}|x_2|^2+x_1y.
\end{equation}
This has a single saddle point at $(0,0,0)$ and corresponds to the optimisation problem
\begin{equation}\label{eq:example-disappearance-optimisation-problem}
\max_{x_1=0}-\frac{1}{2}|x_2|^2
\end{equation}
where the constraint is relaxed via the Lagrange multiplier $y$. The gradient method applied to $\varphi$ is the linear system
\begin{equation}
\begin{bmatrix}
\dot{x}_1\\\dot{x}_2\\\dot{y}
\end{bmatrix}=\begin{bmatrix}
0 & 0&1\\
0 & -1&0\\
-1 &0 &0
\end{bmatrix}
\begin{bmatrix}
x_1\\x_2\\y
\end{bmatrix}
\end{equation}
whose matrix has eigenvalues $-1,\pm i$ so the gradient method is not globally convergent. We again consider the subgradient method on the closed convex set $K_a$ defined by \eqref{eq:K_a} for $a\in\mathbb{R}$ splitting into three cases:
\begin{enumerate}[(i)]
\item $a<0$: As in \autoref{example:projections-break-convergence}(i) the saddle point $(0,0,0)$ lies in the interior of $K_a$. As the unconstrained gradient method is not globally convergent, \autoref{thm:subgradient-method-faces}(ii) implies that the subgradient method on $K_a$ is also not globally convergent.
\item $a>0$: The subgradient method on $K_a$ is given by
\begin{equation}\label{eq:subgradient-method-on-K_a}
\begin{aligned}
\dot{x}_1&=[y]_{x_1-a}^+\\
\dot{x}_2&=-x_2\\
\dot{y}  &=-x_1
\end{aligned}
\end{equation}
The saddle point $(0,0,0)$ lies outside $K_a$. For $(\bar{x}_1,\bar{x}_2,\bar{y})$ to be a $K_a$-restricted saddle point, \eqref{eq:subgradient-method-on-K_a} implies that $\bar{x_1}=\bar{x}_2=0$, but this is impossible in $K_a$, so there are no $K_a$-restricted saddle points. This can also be understood in terms of the optimisation problem \eqref{eq:example-disappearance-optimisation-problem} which has empty feasible set if we impose the further condition that $x_1\ge a>0$. This means that none of our results apply, but a direct analysis of \eqref{eq:subgradient-method-on-K_a} shows that $\dot{y}\le -a<0$ so that $y(t)\to -\infty$ as $t\to\infty$, and the system is not globally convergent.
\item $a=0$: Solving \eqref{eq:subgradient-method-on-K_a} for the $K_0$-restricted saddle points yields the continuum $\{(0,0,y):y\le 0\}$. None of these lie in the interior of $K_0$, {so \color{black}\autoref{thm:subgradient-method-faces}(ii) does not apply and \autoref{thm:subgradient-method-faces}(i) is used to analyze the asymptotic behaviour}. The only proper face of $K_0$ is $F_0$ defined by \eqref{eq:F_a}. On $F_0$, the subgradient method is the system
\begin{equation}
\begin{bmatrix}
\dot{x}_2\\\dot{y}
\end{bmatrix}
=
\begin{bmatrix}
-1&0\\
0&0
\end{bmatrix}
\begin{bmatrix}
x_2\\y
\end{bmatrix}
\end{equation}
together with the equality $x_1=0$. \icl{This} is globally convergent, \icl{since for all initial conditions $x_2$ converges to $0$ and hence we have convergence to a point in the set $\{(0,0,y):y\in\mathbb{R}\}$, which is the set of $F_0$-restricted saddle points} . Therefore the subgradient method on $K_0$ is also globally convergent.
\end{enumerate}
So in this case the subgradient method on $K_a$ starts non-convergent for $a<0$, becomes globally convergent for $a=0$ and finally looses all its equilibrium points when $a>0$.
\end{example}
Although the minimal face $F$ in \autoref{thm:subgradient-method-faces}(i) is given as the intersection of all faces that contain $K$-restricted saddle points, it can be useful to obtain convergence criteria that do not depend upon knowledge of all $K$-restricted saddle points. We note that if the subgradient method is globally convergent on any affine span of a face of $K$, then global convergence is implied.
\begin{corollary}\label{cor:subgradient-method-face-criterion}
\icl{Let function $\varphi$ be $C^2$, and concave-convex on a set $K\subseteq\mathbb{R}^{n+m}$ as defined in \eqref{eq:K_phi}.}
%Let $K=C\times D$, $C\subset\mathbb{R}^n$, $D\subset\mathbb{R}^m,$ where $C,D$ are convex %closed sets.
%Let $K=C\times D\subseteq\mathbb{R}^{n+m}$,  be non-empty, closed and convex.
%Let $\phi:\mathbb{R}^n\times \mathbb{R}^m\rightarrow \mathbb{R}$ be $C^2$ and concave-convex.}
%Let $\varphi$ be $C^2$ and concave-convex on $\mathbb{R}^{n+m}$.
Let $\varphi$ have a $K$-restricted saddle point. Assume that, for any face $F$ of $K$ that contains a $K$-restricted saddle point, the subgradient method on $\affinespan(F)$ is globally convergent. Then the subgradient method on $K$ is globally convergent.
\end{corollary}
\begin{example}\label{example:faces-corresponding-to-positivity-constraints}
To illustrate this result, let us consider the case of positivity constraints, where $(x,y)$ are restricted to $K=\mathbb{R}_+^n\times \mathbb{R}_+^m$. Here the faces of $K$ are given by sets of the form
\begin{equation*}
\{(x,y)\in\mathbb{R}^n_+\times\mathbb{R}^m_+:x_i=0,y_j=0\text{ for }i\not\in I,j\not\in J\}
\end{equation*}
where $I\subseteq\{1,\dotsc,n\}$ and $J\subseteq\{1,\dotsc,m\}$ are sets of indices. The affine span of such a face is then given by
\begin{equation}\label{eq:faces-of-positive-orthant}
\{(x,y)\in\mathbb{R}^{n+m}:x_i=0,y_j=0\text{ for }i\not\in I,j\not\in J\}.
\end{equation}
Thus, by \autoref{cor:subgradient-method-face-criterion}, checking convergence of the subgradient method in this case may be done by checking convergence of the gradient method with any arbitrary set of coordinates fixed as zero\footnote{This result was presented previously by the authors in \cite{Holding-Lestas-CDC2015}.}.
\end{example}
In some cases the faces of the constraint set $K$ have an interpretation in terms of the specific problem.
\begin{example}\label{ex:opt}
Consider the optimisation problem
\begin{equation}\label{eq:example-remove-constraints-optimisation-problem}
\max_{g_j(x)\ge0, j\in\{1,\dotsc,m\}}U(x)
\end{equation}
where $U,g_j:\mathbb{R}^n\to\mathbb{R}$ are concave functions in $C^2$. This is associated with the Lagrangian
\begin{equation}
\varphi(x,y)=U(x)+\sum_{j\in\{1,\dotsc,m\}}y_jg_j(x)
\end{equation}
where $y\in\mathbb{R}^m$ is a vector of Lagrange multipliers\footnote{For simplicity of presentation we shall assume throughout the example that there is no duality gap in the problems considered.}. To ensure that the Lagrange multipliers are non-negative we define the constraint set $K=\mathbb{R}^n\times\mathbb{R}_+^m$. As in \autoref{example:faces-corresponding-to-positivity-constraints} the affine spans of the faces of $K$ are given by \eqref{eq:faces-of-positive-orthant} for $I=\{1,\dotsc,m\}$ and $J$ any subset of $\{1,\dotsc,m\}$. The subgradient method applied on such a face corresponds to the gradient method on the modified Lagrangian
\begin{equation}
\varphi'(x,y)=U(x)+\sum_{j\in J}y_jg_j(x)
\end{equation}
which is associated with the modified optimisation problem
\begin{equation}\label{eq:example-remove-constraints-optimisation-problem-modified}
\max_{g_j(x)=0,j\in J}U(x)
\end{equation}
where, compared to \eqref{eq:example-remove-constraints-optimisation-problem}, the inequality constraints are replaced by equality constraints, and some subset of the constraints are removed.

If $\varphi$ is concave-convex on $\mathbb{R}^{n+m}$
%, which happens if, for example, the constraints $g_j(x)\ge0$ are linear,
{\color{black}then} \autoref{cor:subgradient-method-face-criterion} applies. We obtain that the subgradient method on $K$ applied to $\varphi$ is globally convergent, if, for any $J\subseteq\{1,\dotsc, m\}$, the gradient method applied to the Lagrangian $\varphi'$ corresponding to the modified optimisation problem \eqref{eq:example-remove-constraints-optimisation-problem-modified} is globally {\color{black}convergent.}

%In general, $\varphi$ is only concave-convex on $K$, so \autoref{cor:subgradient-method-face-criterion} does not apply. However, by applying \autoref{thm:subgradient-method-faces} we deduce that the subgradient method is globally convergent, if, for any $J\subseteq\{1,\dotsc, m\}$, every solution of the gradient method applied to the Lagrangian $\varphi'$ corresponding to the modified optimisation problem \eqref{eq:example-remove-constraints-optimisation-problem-modified}, that additionally lies in $K$ for all times $t$, converges to the set of saddle points.

%We note that this result is not sharp, i.e. there are optimisation problems for which the subgradient method is globally convergent, but the gradient method on one of the modified problem defined above fails to converge.
\end{example}

\subsection{A general convergence criterion}\label{subsec:Main:gradmethod}
By combining \autoref{thm:subgradient-method-faces} with the results on the limiting solutions of the (smooth) subgradient method on affine subspaces given in \cite{Holding-Lestas-gradient-method-Part-I} \icl{(recalled in section \ref{app:projected-gradient-method})} we obtain the following convergence criterion for the subgradient method on arbitrary convex sets and arbitrary concave-convex functions. This states that the subgradient method is globally convergent, if it has no trajectory satisfying an explicit linear ODE.

%{\color{black}
%To state the theorem we recall from \cite{Holding-Lestas-gradient-method-Part-I} the definition of the following matrices of partial derivatives of a concave-convex function $\varphi\in C^2$
%%.
%%The matrices $\mathbf{A}(\mathbf{z})$ and $\mathbf{B}(\mathbf{z})$ that will be used in the statement of the theorem were defined in part I of this work \cite{Holding-Lestas-gradient-method-Part-I}, {\color{black}and we include those definitions again below for convenience.
%%In particular, for $\varphi\in C^2$ that is concave-convex
%%we consider the following matrices of partial derivatives of $\varphi$,
%\begin{equation}\label{def-of-A-and-B}
%\begin{aligned}
%\mathbf{A}(\mathbf{z})&=\begin{bmatrix}
%0&\varphi_{xy}(\mathbf{z})\\
%-\varphi_{yx}(\mathbf{z})&0
%\end{bmatrix}\\
%\mathbf{B}(\mathbf{z})&=\begin{bmatrix}
%\varphi_{xx}(\mathbf{z})&0\\
%0&-\varphi_{yy}(\mathbf{z})
%\end{bmatrix}.
%\end{aligned}
%\end{equation}
%}
%For the readers convenience we reproduce both these definitions and the statement of the result needed to prove \autoref{thm:subgrad-method-final-convergence-criterion} in the \appendixref{app:projected-gradient-method}.

The theorem is stated under the assumption that $\mathbf{0}\in K$ is a $K$-restricted saddle point. The general case is obtained by a translation of coordinates.
\begin{theorem}\label{thm:subgrad-method-final-convergence-criterion}
\icl{Let function $\varphi$ be $C^2$, and concave-convex on a set $K\subseteq\mathbb{R}^{n+m}$ as defined in \eqref{eq:K_phi}, and let $\mathbf{0}\in K$ be a $K$-restricted saddle point of $\varphi$.}
%Let $K$ be non-empty, closed and convex in $\mathbb{R}^{n+m}$ with $\mathbf{0}\in K$. Let %$\varphi\in C^2$ be concave-convex on $K$ and have $\mathbf{0}$ as a $K$-restricted saddle %point.
Let $F$ be the minimal face of $K$ that contains all $K$-restricted saddle points and {\color{black}let $V$ be the affine span of $F$.  Let $\mathbf{\Pi}$ be the orthogonal projection matrix onto the orthogonal complement of $N_V$.}
{\color{black}Let also  $\mathbf{A}(.)$ and $\mathbf{B}(.)$ be the matrices defined in \eqref{def-of-A-and-B}.}

%{\color{black}Consider the subgradient method on $K$ applied to $\varphi$ and let $\mathbf{z}(t)$ be a trajectory in its $\omega$-limit set. Then $\mathbf{z}(t)$ {\color{black}satisfies the following}
{\color{black}Then if the subgradient method \eqref{gradmethod-convex-domain} on $K$ applied to $\varphi$ has no non-constant trajectory $\mathbf{z}(t)$ that satisfies both the following
\begin{enumerate}[(i)]
\item
the linear ODE
\begin{equation}\label{eq:projected-gradient-method-linear-ODE}
\dot{\mathbf{z}}(t)=\mathbf{\Pi}\mathbf{A}(\mathbf{0})\mathbf{\Pi}\mathbf{z}(t)
\end{equation}
%and
\item
%the condition,
for all $r\in[0,1]$ {\color{black}and $t\in\mathbb{R}$,}%{\color{black}CHECK THIS}
\begin{equation}\label{eq:first-thm-kernel-condition-projected-gradient-method}
\!\!\mathbf{z}(t)\in\ker(\mathbf{\Pi}\mathbf{B}(r\mathbf{z}(t))\mathbf{\Pi})\cap \ker(\mathbf{\Pi}(\mathbf{A}(r\mathbf{z}(t))-\mathbf{A}(\mathbf{0}))\mathbf{\Pi}),
\end{equation}
\end{enumerate}
then the subgradient method is globally convergent.}

%{\color{black}Consider the subgradient method on $K$ applied to $\varphi$ and let $\mathbf{z}(t)$ be a trajectory in its $\omega$-limit set. Then $\mathbf{z}(t)$ {\color{black}satisfies the following}
%%Then if the subgradient method on $K$ applied to $\varphi$ has no non-constant trajectory $\mathbf{z}(t)$ {\color{black}that satisfies both of the following
%\begin{enumerate}[(i)]
%\item
%the linear ODE
%\begin{equation}\label{eq:projected-gradient-method-linear-ODE}
%\dot{\mathbf{z}}(t)=\mathbf{\Pi}\mathbf{A}(\mathbf{0})\mathbf{\Pi}\mathbf{z}(t)
%\end{equation}
%%and
%\item
%%the condition,
%for all $r\in[0,1]$ {\color{black}and $t\in\mathbb{R}$,}%{\color{black}CHECK THIS}
%\begin{equation}\label{eq:first-thm-kernel-condition-projected-gradient-method}
%\!\!\mathbf{z}(t)\in\ker(\mathbf{\Pi}\mathbf{B}(r\mathbf{z}(t))\mathbf{\Pi})\cap \ker(\mathbf{\Pi}(\mathbf{A}(r\mathbf{z}(t))-\mathbf{A}(\mathbf{0}))\mathbf{\Pi}),
%\end{equation}
%\end{enumerate}
%
%Therefore if the subgradient method has non non-constant trajectory $\mathbf{z}(t)$ that satisfies both \eqref{eq:projected-gradient-method-linear-ODE}, \eqref{eq:first-thm-kernel-condition-projected-gradient-method},
%then the subgradient method is globally convergent.}
\end{theorem}
\icl{The proof of \autoref{thm:subgrad-method-final-convergence-criterion} is provided in Appendix \ref{sec:conv_subgradient-method_proof}.}
\begin{remark}\label{rem:proj}
Although the condition \eqref{eq:first-thm-kernel-condition-projected-gradient-method} appears difficult to verify, it is only necessary to show that the condition \emph{does not} hold {\color{black}(by non-trivial trajectories)} in order to prove global convergence. This turns out to be easy in many cases, for example in the proofs of the convergence of the modification methods {\color{black}discussed in section \ref{sec:modification-methods}} (\autoref{thm:convergence-of-modification-methods}). \icl{In particular, these are examples where global convergence is desired to a saddle point without knowing the saddle points a priori
%\footnote{\icl{Note that such cases where saddle points are not explicitly known also relate to \autoref{cor:subgradient-method-face-criterion}. In particular, since subspace $V=\text{aff}(F)$ is not known, convergence needs to be shown for the subgradient method on a general affine subspace.}}
and without function $\varphi$ being strictly concave convex. The derivations exploit the structure of the matrices $\mathbf{A}, \mathbf{B}, \mathbf{\Pi}$ to prove that  \eqref{eq:projected-gradient-method-linear-ODE} and  \eqref{eq:first-thm-kernel-condition-projected-gradient-method} are satisfied only by saddle points.}
\end{remark}
{\color{black}
\begin{remark}\label{rem:OmLinear}
It should be noted that \eqref{eq:projected-gradient-method-linear-ODE} and  \eqref{eq:first-thm-kernel-condition-projected-gradient-method} are satisfied by all trajectories $\mathbf{z}(t)$ in the $\omega$-limit set of the subgradient method. This follows from Theorem \ref{thm:subgradient-method-faces} and Theorem~\ref{thm:projected-gradient-method-result} and is stated in the corollary below \icl{(proved in Appendix \ref{sec:conv_subgradient-method_proof})}.
%As in \autoref{thm:subgrad-method-final-convergence-criterion} we assume, without loss of generality, that $\mathbf 0$ is a $K$-restricted saddle point.
\end{remark}
\begin{corollary}\label{cor:subgradient_linear}
%Let $K\subseteq\mathbb{R}^{n+m}$ be closed and convex, $\varphi$ be $C^2$ be concave-convex on %$K$, and $\mathbf 0$ as a $K$-restricted saddle point.
Consider the subgradient method \eqref{gradmethod-convex-domain} and let  $\mathbf 0$ be a $K$-restricted saddle point.
%on $K$ applied to $\phi$. %\autoref{thm:subgrad-method-final-convergence-criterion}.
Then any trajectory $\mathbf{z}(t)$ in the $\omega$-limit set satisfies \eqref{eq:projected-gradient-method-linear-ODE} and  \eqref{eq:first-thm-kernel-condition-projected-gradient-method}, i.e. it is a solution of a linear ODE.
\end{corollary}
}
\section{Applications}\label{sec:modification-methods}

In this section we apply the results of \autoref{sec:Main} to obtain global convergence in a number cases.
%First we consider the subgradient method {\color{black}applied to a strictly concave-convex %function on an arbitrary convex domain}. {\color{black}Then we}
\icl{In particular, we} look at %examples
\icl{examples} of
{\color{black} modification methods, relevant in network optimization, where the concave-convex function is modified to provide guarantees of convergence.}
% {\color{black}modification methods in network optimization that provide guarantees of convergence.}
%modifying a concave-convex function to obtain convergence, {\color{black}which are relevant in %network optimization}.
{\color{black}The application of one such modification method to the problem {\color{black}of multi-path routing is also \icl{discussed}.}}
% in  Appendix \ref{subsec:examples:multi-path-routing}.}}

The proofs for this section are provided in appendix \ref{sec:proofs-of-examples}.

\subsection{Modification methods for convergence}\label{subsec:modification-methods}
We will consider methods for modifying $\varphi$ so that the (sub)gradient method converges to a saddle point. {\color{black}The methods that will be discussed are relevant in network optimisation (see e.g. \cite{arrow}, \cite{Paganini}), as they }%where it is important to
preserve the localised structure of the {\color{black}dynamics. It should be noted that these modifications do not necessarily render the {\color{black}function} strictly concave-convex and hence convergence proofs are more involved. We show below that the results in \autoref{sec:Main} provide a systematic and unified way of proving convergence by making use of Theorem \ref{thm:subgrad-method-final-convergence-criterion}, while also allowing to consider these methods in a generalized setting of a general convex \icl{domains for the variables $x, y$ respectively, in the concave-convex function $\varphi(x,y)$.}}
%, which makes the use of higher order information difficult.
%One such method was described in part I of this work \cite[{\color{black}modification methed section}]{Holding-Lestas-gradient-method-Part-I} for the gradient method. We will extend this method to the subgradient method, describe two more such methods, and then give {\color{black}corresponding} convergence results.
\subsubsection{Auxiliary variables method}\label{subsec:modification-method}
%In part I of this work we described a modification method for the gradient method \cite[{\color{black}modification methed section}]{Holding-Lestas-gradient-method-Part-I}. We now recall this method and extend it to the subgradient method restricted to an arbitrary convex domain. We refer the reader to \cite[{\color{black}modification methed section}]{Holding-Lestas-gradient-method-Part-I} for some additional discussion. In \autoref{subsec:examples:multi-path-routing} below we give an example application of this method to the problem of multi-path congestion control.

Given a concave-convex function $\varphi$ defined on a convex domain $K$ \icl{as in \eqref{eq:K_phi}}, we define the modified concave-convex function $\varphi':\mathbb{R}^{n'}\times K\to\mathbb{R}$ as
{\color{black}
\begin{equation}\label{modified-varphi}
\begin{aligned}
& \quad \varphi'(x',x,y)=\varphi(x,y)+\psi(Mx-x')\\
&\psi:\mathbb{R}^{n'}\to\mathbb{R},
\psi\in C^2, \text{ is strictly concave}\\ % with } \psi(0)=0, \psi(u)\leq0\\
&\qquad \text{with } \psi(0)=0, \psi(u)\leq0,
%maximum at } \psi(0)=0,
%,\\
%&M\in\mathbb{R}^{n'\times n}\text{ is a constant matrix that satisfies}\\ &\qquad\ker(M)\cap\ker(\varphi_{xx}(\mathbf{\bar{z}}))=\{0\}\\
%&\quad\text{for a $K$-restricted saddle point } \mathbf{\bar{z}} \text{ of } \varphi.
\end{aligned}
\end{equation}
where $x'$ is a vector of $n'$ auxiliary variables, and $M\in\mathbb{R}^{n'\times n}$ is a constant matrix that satisfies $\ker(M)\cap\ker(\varphi_{xx}(\mathbf{\bar{z}}))=\{0\}$ for a $K$-restricted saddle point  $\mathbf{\bar{z}}$ of $\varphi$. }

 We define the augmented convex domain as $K'=\mathbb{R}^{n'}\times K$. Note that the additional auxiliary variables are not restricted and are allowed to take values in the whole of $\mathbb{R}^{n'}$. {\color{black}Also note that the} $n\times n$ identity matrix always satisfies the assumptions upon $M$ above.

%Note that
{\begin{remark}
{\color{black}An important feature of this modification (and also the ones that will be considered below) is the fact that there is} a correspondence between $K$-restricted saddle points of $\varphi$ and $K'$-restricted saddle points of $\varphi'$, {\color{black}with the values of $x,y$ at the saddle points remaining unchanged. In particular, if} $(\bar{x},\bar{y})$ is a $K$-restricted saddle point of $\varphi$, then $(M\bar{x},\bar{x},\bar{y})$ is a $K'$-restricted saddle point of $\varphi'$. In the reverse direction, if $(\bar{x}',\bar{x},\bar{y})$ is a $K'$-restricted saddle point of $\varphi'$ then $M\bar{x}=\bar{x}'$ and $(\bar{x},\bar{y})$ is a $K$-restricted saddle point of $\varphi$.
\end{remark}
{\color{black}
\begin{remark}
The significance of this method will become more {\color{black}clear in
%\autoref{subsec:examples:multi-path-routing}
%where an application to
the multipath routing problem discussed in Appendix~\ref{subsec:examples:multi-path-routing}}. In particular, {\color{black}this method allows convergence to be guaranteed in network optimization problems} without introducing additional information transfer among nodes. Special cases of this method have also been used
%\footnote{\color{black}Auxiliary variables are also used in discrete time in proximal point %methods \cite{Rockafellar_Proximal}.}
in \cite{Eoin}, \cite{holdingAutomatica} {\color{black}in applications} in economic and power networks.
\end{remark}
}

\subsubsection{Penalty function method}\label{subsec:examples:penalty-function}
For this and the next method we will assume that the concave-convex functions $\varphi$ is a Lagrangian originating from a concave optimization problem (see subsection \ref{sec:application-to-concave-optimization}). We will assume that the Lagrangian $\varphi$ satisfies
\begin{equation}\label{lag-conds-fullspace}
\begin{aligned}
\varphi(x,y)&=U(x)+y^Tg(x)\\
C^2\ni U&:\mathbb{R}^n\to\mathbb{R}\text{ is concave}\\
C^2\ni g&:\mathbb{R}^n\to\mathbb{R}^m\text{ is concave}.
\end{aligned}
\end{equation}

We consider a so called penalty method (see e.g. \cite{Freund}). This method adds a penalising term to the Lagrangian based directly on the constraint functions. The new Lagrangian $\varphi'$ is defined by
\begin{equation}\label{penalty-function-method}
\begin{aligned}
\varphi'(x,y)&=\varphi(x,y)+\psi(g(x))\\
C^2\ni \psi:\mathbb{R}^m&\to\mathbb{R}\text{ is strictly concave with }\psi_u>0\\
\psi(u)&=0\iff u\ge0.
\end{aligned}
\end{equation}
It is easy to see that the saddle points of $\varphi$ and $\varphi'$ are the same.

{\color{black}
\begin{remark}\label{rem:penlty_decentr}
This modification method
is also often applied to network optimization problems, i.e. problems where $U(x)$ is of the form $U(x)=\sum_iU_i(x)$ and each of the $U_i(x)$ is a function of only a few of the components of $x$. Similarly each component, $g_i(x)$, of the constraints $g(x)$ depends on only a few of the components of $x$. The subgradient method for such problems applied to  \eqref{lag-conds-fullspace} has a decentralized structure. When applied to the modified version \eqref{penalty-function-method} the dynamics will still have a decentralized structure, but will often also involve additional information exchange between neighboring nodes, e.g. when $g(x)$ is linear, due to the nonlinearity of the function~$\psi(.)$.
\end{remark}
}
%, when applied (with proper choice of $\psi$) to distributed optimization problems, does not destroy the local nature of the gradient method, and implementation is possible with only minimal additional information transfer. In particular the additional transfer is only between neighbouring nodes.

\begin{remark}
This method has been considered previously by many authors, (see \cite{Paganini} and the references {\color{black}therein\footnote{\color{black}Note that a related modification method in discrete time is the ADMM method \cite{BoydADMM}, \cite{ADMM_Gabay}.}}), either without constraints, or with positivity constraints, i.e. $K=\mathbb{R}^n_+\times\mathbb{R}^m_+$. \autoref{thm:convergence-of-modification-methods} below applies to all non-empty closed {\color{black} sets} $K\subseteq \mathbb{R}^{n+m}$ \icl{which are a product set
%$K=K_x\times K_y$
of convex sets as in \eqref{set:K}}.
\end{remark}

\subsubsection{Constraint modification method}\label{subsec:examples:constraint-modification}
We next recall a method proposed by Arrow et al.\cite{arrow} and later studied in \cite{Paganini}. Here we instead modify the constraints to enforce strict concavity. The Lagrangian \eqref{lag-conds-fullspace} is modified to become:
\begin{equation}\label{constraint-modification-method-assumptions}
\begin{aligned}
\varphi'(x,y)&=U(x)+y^T\psi(g(x))\\
C^2\ni U&:\mathbb{R}^n\to\mathbb{R}\text{ is concave}\\
C^2\ni g&:\mathbb{R}^n\to\mathbb{R}^m\text{ is concave}\\
C^2\ni\psi=&[\psi^1,\dotsc,\psi^m]^T:\mathbb{R}^m\to\mathbb{R}^m\\
\psi^j(0)&=0, \psi^j_u\ge0\text{ and }\psi^j_{uu}<0 \text{ for }j=1,\dotsc m.
\end{aligned}
\end{equation}
It is clear that the {\color{black}value of $x$ at the} saddle points of the modified and original Lagrangian will be the same.
{\color{black}
In analogy with {\color{black}Remark} \ref{rem:penlty_decentr}, this method also preserves the decentralized structure
of the subgradient method for network optimization problems, but may require additional information transfer.}
%Again, the method preserves the local structure of the gradient method, requiring only minimal %additional information transfer.

\begin{remark}
Previous works \cite{arrow},\cite{Paganini},\cite{Cherukuri-primal-dual} have proved convergence of this method with positivity constraints, i.e. $K=\mathbb{R}_+^n\times\mathbb{R}^m_+$. \autoref{thm:convergence-of-modification-methods} below applies to any constraint \icl{set $K$ which is a product set of convex sets as in \eqref{set:K}}.
%which is a product set $K=K_x\times K_y$ with $K_x\subseteq\mathbb{R}^n$, %$K_y\subseteq\mathbb{R}^m$ both non-empty closed and convex. %In particular, it applies to the Lagrangian obtained from \eqref{primal} where the $x\in C$ constraint is not relaxed and the $g(x)\ge0$ constraint is relaxed via Lagrangian multipliers which are forced to be positive.
\end{remark}
\icl{
\begin{remark}\label{rem:meth_comp}
It should be noted that even though the three methods described in this section all provide global convergence guarantees, they lead to different information structures in the underlying dynamics when applied to network optimization problems. In particular, the auxiliary variable method leads to fully decentralized implementations, whereas the other two can require additional information transfer among nodes. This will be illustrated in \autoref{subsec:examples:multi-path-routing} where the multipath routing problem will be discussed.
%In particular, when matix $M$ in the auxiliary variable method is chosen as the identity matrix, and $\phi(x,y)$ is the Lagrangian of a network optimizaton problem, then the modified (sub)gradient dynamics do not require an additional information transfer relative to the unmodified dynamics. This will be illustrated in \autoref{subsec:examples:multi-path-routing} where the multipath routing problem will be discussed. This is in contrast with the other two modification methods where the nonlinearity of $\psi$ implies that the update rules in the modified (sub)gradient dynamics, will require additional information transfer among nodes, relative to the unmodified (sub)gradient, when function $g$ associated with the constraints is linear (see also \autoref{subsec:examples:multi-path-routing}).
\end{remark}
}

\subsubsection{Convergence results}
We now give a global convergence result for each of the methods described above on {\color{black}general} convex domains.

%\label{set:K}

\icl{
\begin{theorem}[Convergence of modification methods]\label{thm:convergence-of-modification-methods}
Let $K\subseteq\mathbb{R}^{n+m}$ be a non-empty closed set that is the cartesian product of two convex sets as in \eqref{set:K}, and %$\varphi\in C^2$ be concave-convex on $\mathbb{R}^{n+m}$ with a $K$-restricted saddle point. Assume one of the following:
assume that $\varphi$, $\varphi'$ satisfy one of the following:
\begin{enumerate}
\item \emph{Auxiliary variable method:} Let $\varphi\in C^2$ be concave-convex on $K$ and %$K\subseteq\mathbb{R}^{n+m}$ a non-empty closed convex set. Let
    $\varphi'$, $K'$ be defined by \eqref{modified-varphi} and the text directly below it.
\item \emph{Penalty function method:} Let $\varphi$ have the form \eqref{lag-conds-fullspace}, and $\varphi'$ be defined by \eqref{penalty-function-method}. % and $K\subseteq\mathbb{R}^{n+m}$ be an arbitrary %non-empty closed convex set.
\item \emph{Constraint modification method:} Let $\varphi$ have the form \eqref{lag-conds-fullspace}, and $\varphi'$ be given by \eqref{constraint-modification-method-assumptions}.
    %and $K=K_x\times K_y$ with $K_x\subseteq\mathbb{R}^n$, $K_y\subseteq\mathbb{R}^m$ both %non-empty closed and convex.
\end{enumerate}

%\begin{theorem}[Convergence of modification methods]\label{thm:convergence-of-modification-methods}
%%Let $K\subseteq\mathbb{R}^{n+m}$ be non-empty closed and convex, and $\varphi\in C^2$ be concave-convex on $\mathbb{R}^{n+m}$ with a $K$-restricted saddle point. Assume one of the following:
%Assume that $\varphi$, $\varphi'$ and $K$ satisfy one of the following:
%\begin{enumerate}
%\item \emph{Auxiliary variable method:} Let $\varphi\in C^2$ be concave-convex on $K\subseteq\mathbb{R}^{n+m}$ a non-empty closed convex set. Let $\varphi'$ and $K'$ be defined by \eqref{modified-varphi} and the text directly below it.
%\item \emph{Penalty function method:} Let $\varphi$ have the form \eqref{lag-conds-fullspace}, $\varphi'$ be defined by \eqref{penalty-function-method} and $K\subseteq\mathbb{R}^{n+m}$ be an arbitrary non-empty closed convex set.
%\item \emph{Constraint modification method:} Let $\varphi$ have the form \eqref{lag-conds-fullspace}, $\varphi'$ be given by \eqref{constraint-modification-method-assumptions} and $K=K_x\times K_y$ with $K_x\subseteq\mathbb{R}^n$, $K_y\subseteq\mathbb{R}^m$ both non-empty closed and convex.
%\end{enumerate}
{Also assume that $\varphi$ has a $K$-restricted saddle point. Then the subgradient method \eqref{gradmethod-convex-domain} applied to $\varphi'$ on domain $K'$ in $1)$ and domain $K$ in $2), 3)$}  is globally convergent.
\end{theorem}}
{\color{black}
\begin{remark}
Each of the convergence results in \autoref{thm:convergence-of-modification-methods} is proved using \autoref{thm:subgrad-method-final-convergence-criterion}.
It should also be noted that none of the modification methods produce necessarily  a \emph{strictly} concave-convex function $\varphi'$. Global convergence to a saddle point is still though guaranteed by ensuring that no trajectory, other than saddle points, satisfy conditions \eqref{eq:projected-gradient-method-linear-ODE}, \eqref{eq:first-thm-kernel-condition-projected-gradient-method} in \autoref{thm:subgrad-method-final-convergence-criterion}.
%Despite the complexity of the non-linear and non-smooth systems of ODEs involved, the application of \autoref{thm:subgrad-method-final-convergence-criterion} makes the convergence easy to verify, which is evident from the simplicity of the {\color{black}proofs.}
\end{remark}}

\subsection{Multi-path congestion control}\label{subsec:examples:multi-path-routing}
{\color{black}In this \icl{section} %appendix
we discuss multipath routing as an example of a saddle point problem that is inherently not strictly concave-convex, but where convergence can be guaranteed via appropriate modifications that do not necessarily render the problem strictly concave-convex.}}

{\color{black}Combined control of routing and flow} is a problem that has received considerable attention within the communications literature due to the significant advantages it can provide relative to congestion control algorithms that use single paths \cite{Lee2002survey}. Nevertheless {\color{black}its} implementation is not directly obvious as the availability of multiple routes can render the network prone to route flapping instabilities \cite{Wang1992analysis}, \cite{Kar01optimizationbased}, \cite{Voice2007}.

A classical approach to analyse such algorithms is to formulate them as solving a corresponding network optimization problem \cite{KMT}, \cite{SrikantB} with primal/dual update rules leading to a decentralized implementation. This optimization problem is, however, not strictly concave and modifications that make it strictly concave can lead to a deviation from the optimal solution \cite{Voice2007}, \cite{Feng-automatica}.

%where aggregate user utilities are maximized subject to capacity constraints \cite{KMT}.

%A classical approach to analyse such algorithms is to formulate them as solving a network optimization problem where aggregate user utilities are maximized subject to capacity constraints \cite{KMT}. In the seminal work in \cite{KMT} it was noted that when capacity constraints are relaxed with penalty functions and primal algorithms are considered, then convergence can  be  guaranteed despite the presence of multiple routes. In order, however, to achieve the network capacities, dual or primal/dual algorithms need to be deployed \cite{SrikantB}. Nevertheless, when multiple routes per source/destination pair are available the corresponding optimization problem is known to be not strictly convex and the use of classical gradient dynamics can lead to unstable behaviour \cite{Voice2007}, \cite{Kar01optimizationbased}, \cite{arrow}. In order to address this issue various studies have considered relaxations that lead to a modified optimization problem that is strictly convex \cite{Voice2007}, \cite{Feng-automatica}. This leads to algorithms with guaranteed convergence, but with the equilibrium solution deviating from that of the solution of the original optimization problem.

Here we consider a multi-path routing problem with a fixed number of routes per source/destination pair, as in \cite{KMT}, \cite{Voice2007}, \cite{Lestas-Routing-CDC2004}, \cite{Lin-Shroff}. For such schemes we investigate algorithms that allow the corresponding network optimization problem to be solved without requiring any relaxation in its solution or any additional information {\color{black}exchange.}
%In particular, we show that this is feasible by incorporating appropriate higher order dynamics in the local update rules.

\subsubsection{Problem formulation}
%We consider a multi-path routing problem where each source/destination pair has a fixed number %of routes.
%In particular,
{\color{black}We consider} a network that consists of sources $s_1,\dotsc,s_m$, routes $r_1,\dotsc,r_n$, and links $l_1,\dotsc,l_l$. Each source $s_i$ is associated with a unique destination for a message which is to be routed {\color{black}each source also has a fixed set of routes associated with it.}. Every route $r_j$ has a unique source $s_i$, and we write $r_j\sim s_i$ to mean that $s_i$ is the source associated with route $r_j$. Routes $r_j$ each use a number of links, and we write $r_j\sim l_k$ to mean that the link $l_k$ is used by the route $r_j$. The desired running capacity of the link $l_k$ is denoted $C_k$, and $0\le C\in\mathbb{R}^l$ is the vector of these capacities. We let $A$ be the connectivity matrix, so that $A_{kj}=1$ if $l_k\sim r_j$ and $0$ otherwise. In the same way we set $H_{ij}=1$ if $s_i\sim r_j$ and $0$ otherwise. $x_j$ denotes the current usage of the route $r_j$. We associate to each source $s_i$ a strictly concave, increasing utility function $U_i$.

We consider the problem of maximising total utility
%over the network  %stated as
\begin{equation}\label{eq:multi-path-congestion-optimization-problem}
\max_{x\ge0,Ax\le C} \sum_{s_i}U_i\left(\sum_{r_j\sim s_i}x_j\right).
\end{equation}
Here the first sum is over all sources $s_i$, and the second over routes $r_j$ with $r_j\sim s_i$ (we shall use such notation throughout {\color{black}this section}). This optimisation problem is associated with the Lagrangian
\begin{equation}\label{eq:multi-path-congestion-lagrangian}
\varphi(x,y)=\sum_{s_i}U_i\left(\sum_{r_j\sim s_i}x_j\right)+y^T(C-Ax).
\end{equation}
where $y\in\mathbb{R}^l_+$ are {\color{black}the Lagrange multipliers.
% that relax the $Ax\le C$ constraint.
Note that even though $U_i(.)$ is strictly concave, this is not strictly concave in \eqref{eq:multi-path-congestion-optimization-problem} with respect to the decision variables $x_i$, hence the Lagrangian $\varphi(x,y)$ is not strictly concave-convex.}

A common approach in the context of congestion control is to consider primal-dual dynamics originating from this Lagrangian so as to deduce decentralized algorithms for solving the network optimisation problem \eqref{eq:multi-path-congestion-optimization-problem} \cite{KMT},\cite{SrikantB}.
%. This has the advantage of producing a decentralised algorithm\cite{KMT}.
This gives rise to the subgradient {\color{black}dynamics}
{%\small
\begin{equation}\label{eq:multi-path-routing-dynamics}
\begin{aligned}
\dot{x}_j&=\left[U_i'\left(\sum_{s_i\sim r_k}x_k\right)-\sum_{l_k\sim r_j}y_k\right]_{x_j}^+ \\
\dot{y}_k&=\left[\sum_{l_k\sim r_j}x_j-C_k\right]_{y_k}^+
\end{aligned}
\end{equation}}
where $s_i\sim x_j$ in the equation for $\dot{x}_j$ and $U'_i$ is the derivative of the utility function $U_i$. Note that the equilibrium points of \eqref{eq:multi-path-routing-dynamics} are saddle points of the Lagrangian (under the positivity constraints on $x$ and $y$) and hence also solutions of the optimization problem \eqref{eq:multi-path-congestion-optimization-problem} (Slater's condition is assumed to {\color{black}hold).} % throughout the paper).

\begin{remark}
	The dynamics \eqref{eq:multi-path-routing-dynamics} {\color{black}are
%nothing other than
the subgradient method \eqref{gradmethod-convex-domain} applied to the Lagrangian \eqref{eq:multi-path-congestion-lagrangian} on the positive orthant $\mathbb{R}^{n+l}_+$.}
\end{remark}

The dynamics \eqref{eq:multi-path-routing-dynamics} are also localised in the sense that the update rules for $x_j$ depend only on the current usage, $x_k$, of routes with the same source and of the congestion signals associated with links on these routes. In the same way the update rules for congestion signals $y_k$ depend only on the usage of routes using the associated link.
\subsubsection{Instability}
The dynamics \eqref{eq:multi-path-routing-dynamics} inherit the stability properties of the subgradient method discussed in \autoref{sec:Main}. In particular the distance of $(x(t),y(t))$ from any saddle point $(\bar{x},\bar{y})$ is non-increasing. However, the lack of strict concavity of the Lagrangian \eqref{eq:multi-path-congestion-lagrangian} leads to a lack of global convergence of the dynamics \eqref{eq:multi-path-routing-dynamics} in some situations as we shall describe below.

%To simplify the situation we shall
{\color{black}We assume for simplicity that} there is a strictly positive saddle point $\bar{\mathbf{z}}>0$. In this situation \autoref{thm:subgradient-method-faces}(ii) applies, and the convergence properties are the same as {\color{black}those of} the unconstrained gradient method. The structure of the problem suggests an application of \cite[{\color{black} %\autoref{I-S-inside-Slinear}}]
Theorem 21}]
{Holding-Lestas-gradient-method-Part-I}. Here a simple computation yields that $\mathcal{S}_{\text{linear}}$ is equal to $\bar{\mathcal{S}}$ (we use the notation of \cite{Holding-Lestas-gradient-method-Part-I}) unless the following \emph{algebraic} condition on the network topology holds:
 \begin{equation}
 \label{eq:algebriac-instability-condition}
 \exists u\in \ker(H)\setminus\{0\},\lambda>0\text{ such that }A^TAu=\lambda u.
 \end{equation}
 \cite[{\color{black}
 Theorem 21}]
 %\autoref{I-S-inside-Slinear}}]
 {Holding-Lestas-gradient-method-Part-I} tells us that global convergence holds if \eqref{eq:algebriac-instability-condition} does not hold, but in fact more is true.
 \begin{proposition}\label{prop:algebriac-instability}
Let $\mathbf{\bar{z}}=(\bar{x},\bar{y})>0$ be a saddle point of $\varphi$ defined by \eqref{eq:multi-path-congestion-lagrangian} and $U_i\in C^2$ be be strictly concave and strictly increasing. Then the dynamics \eqref{eq:multi-path-routing-dynamics} are globally convergent if and only if \eqref{eq:algebriac-instability-condition} does not hold.
\end{proposition}

The algebraic criterion \eqref{eq:algebriac-instability-condition} on the network topology is satisfied by many networks, for example the network in \autoref{fig:network}.

We also remark that under the condition \eqref{eq:algebriac-instability-condition}, the system is sensitive to noise in the sense that the unconstrained dynamics satisfy the conditions of \cite[{\color{black}
Theorem 22}]
%\autoref{I-prop:unstable-to-noise}}]
{Holding-Lestas-gradient-method-Part-I}.
\subsubsection{Modified dynamics}
\label{sec:modified}
Here we present a modification of the dynamics \eqref{eq:multi-path-routing-dynamics}, that, while still fully localised, {\color{black}gives} guaranteed convergence to an optimal solution of \eqref{eq:multi-path-congestion-optimization-problem}.

We use the auxiliary variables method described in \autoref{subsec:modification-method} {\color{black}and} define {\color{black}the modified} optimisation problem
\begin{equation}\label{eq:multi-path-congestion-modified-optimization-problem}
\max_{\substack{x\ge0,x'\in\mathbb{R}^n\\Ax\le C}} \sum_{s_i}U_i\left(\sum_{r_j\sim s_i}x_j\right)-\frac12\sum_{r_k}\kappa_k|x_k'-x_k|^2
\end{equation}
where $x'\in\mathbb{R}^n$ is an additional vector to be optimised over, and $\kappa_k>0$ are arbitrary constants. It is important to note that this has the same optimal $x$ points as \eqref{eq:multi-path-congestion-optimization-problem}. This gives rise to a modified Lagrangian
\begin{equation}\label{eq:multi-path-routing-modified-lagrangian}
\begin{aligned}
\varphi'(x',x,y)&=\sum_{s_i}U_i\left(\sum_{r_j\sim s_i}x_j\right)+y^T(C-Ax)\\
&\quad-\frac12\sum_{r_k}\kappa_k|x'_k-x_k|^2.
\end{aligned}
\end{equation}
The new dynamics are given by the following subgradient method.
\begin{equation}\label{eq:multi-path-routing-modified-dynamics}
\begin{aligned}
\dot{x}_j&=\left[U_i'\left(\sum_{s_i\sim r_k}x_k\right)\!-\!\sum_{l_k\sim r_j}y_k+\kappa_j(x_j'-x_j)\right]_{x_j}^+\\
\dot{x}_j'&=\kappa_j(x_j-x_j')\\
\dot{y}_k&=\left[\sum_{l_k\sim r_j}x_j-C_k\right]_{y_k}^+.
\end{aligned}
\end{equation}
\begin{remark}
	The dynamics \eqref{eq:multi-path-routing-modified-dynamics} are the subgradient method \eqref{gradmethod-convex-domain} {\color{black}applied to the modified Lagrangian \eqref{eq:multi-path-routing-modified-lagrangian} on $\mathbb{R}^n_+\times\mathbb{R}^n\times\mathbb{R}^l_+$}. The Lagrangian \eqref{eq:multi-path-routing-modified-lagrangian} corresponds to \eqref{modified-varphi} with $\psi(z)=-|z|^2/2$ and $M$ the $n\times n$ identity matrix.
\end{remark}

It is apparent (as discussed in subsection \ref{subsec:modification-method}) that the equilibrium points of the modified dynamics \eqref{eq:multi-path-routing-modified-dynamics} and the original dynamics \eqref{eq:multi-path-routing-dynamics} are in correspondence. We remark that the new dynamics are analogous to the addition of a low pass filter to the unmodified dynamics \eqref{eq:multi-path-routing-dynamics}.

These dynamics are still localised. Each route $r_k$ is now associated with its usage, $x_k$, and a new variable $x_k'$. To update $x_k$ the only additional information required over the unmodified scheme is the value of $x_k'$, and to update $x_k'$ one only needs $x_k$. Thus the new variables $x_k'$ are local to the updaters of $x_k$.

\icl{It should be noted that if instead the other two modification methods described in \autoref{subsec:modification-methods} were used, then the modified gradient dynamics would require additional information transfer among nodes for their implementation. In particular, due to the nonlinearity of the function $\psi$ in \eqref{penalty-function-method}, \eqref{constraint-modification-method-assumptions}, the ODE for the $x_i$ updates would have required also the flows $x_j$ from neighboring nodes, which is practically undesirable as such information is not available in existing implementations of congestion control algorithms.}

Convergence of the modified dynamics \icl{\eqref{eq:multi-path-routing-modified-dynamics}} to an optimum of the original problem now follows immediately from \autoref{thm:convergence-of-modification-methods}.1).
\begin{proposition}\label{thm:multi-path-routing-modified-method-converges}
	Let $U_i\in C^2$ be strictly concave and strictly increasing. Then solutions of \eqref{eq:multi-path-routing-modified-dynamics} converge as $t\to\infty$ to maxima of the original problem \eqref{eq:multi-path-congestion-optimization-problem}.
\end{proposition}
\begin{remark}
	The use of derivative action to damp oscillatory behaviour has been studied previously in the context of node based multi-path routing in \cite{Paganini-Mallada} by incorporating derivative action in a price signal that gets communicated (i.e. a form of prediction is needed) and a local stability result was derived. This has also been used in gradient dynamics in game theory in \cite{Shamma-Dynamic}. A control scheme similar to \eqref{eq:multi-path-routing-modified-dynamics} for multi-path routing was proposed in \cite{Lin-Shroff} and studied in discrete and continuous time. In \cite{Lin-Shroff} the scheme differs from \eqref{eq:multi-path-routing-modified-dynamics} in that the $x_j$ variables are updated instantaneously. In our context this would be
	\begin{equation}
	x(t)=\operatorname*{argmax}_{x\ge0, Ax\le C}\varphi'(x'(t),x,y(t)).
	\end{equation}
%	A distinctive feature of our work is that our proof explicitly takes the switching nature of the system into account (see \autoref{prop:subgradient-method-positivity-constraints-criterion} below).
\end{remark}

\begin{figure}[!h]
	\centering
\resizebox{2cm}{2cm}{%
	\begin{tikzpicture}[auto,node distance=2cm,
	thick,
	source/.style={circle,fill=blue!20,draw,font=\sffamily\Large\bfseries},
	dest/.style={circle,fill=red!20,draw,font=\sffamily\Large\bfseries},
	link1/.style={dotted,->},
	link2/.style={->},scale=0.50]
	\node[source] (1) at (0,0) {$1$};
	\node[dest] (3) [below of=1] {$3$};
	\node[source] (2) [right of=1] {$2$};
	\node[dest] (4) [below of=2] {$4$};
	\draw [link1] (1) to [bend right=15] (2);
	\draw [link1] (2) to [bend right=15] (4);
	\draw [link1] (1) to [bend left=15] (3);
	\draw [link1] (3) to [bend left=15] (4);
	\draw [link2] (2) to [bend right=15] (1);
	\draw [link2] (1) to [bend right=15] (3);
	\draw [link2] (2) to [bend left=15] (4);
	\draw [link2] (4) to [bend left=15] (3);
	\end{tikzpicture}
}
	\caption{{\color{black}First} example network. Sources at $1$ and $2$ transmit to the destinations $4$ and $3$ respectively. Each has a choice of two routes. Routes associated with the source at $1$ are dotted lines, while those associated with the source at $2$ are solid lines.}
	\label{fig:network}
\end{figure}

\begin{figure}[!h]
\centering	\includegraphics[keepaspectratio=true,width=0.4\textwidth]{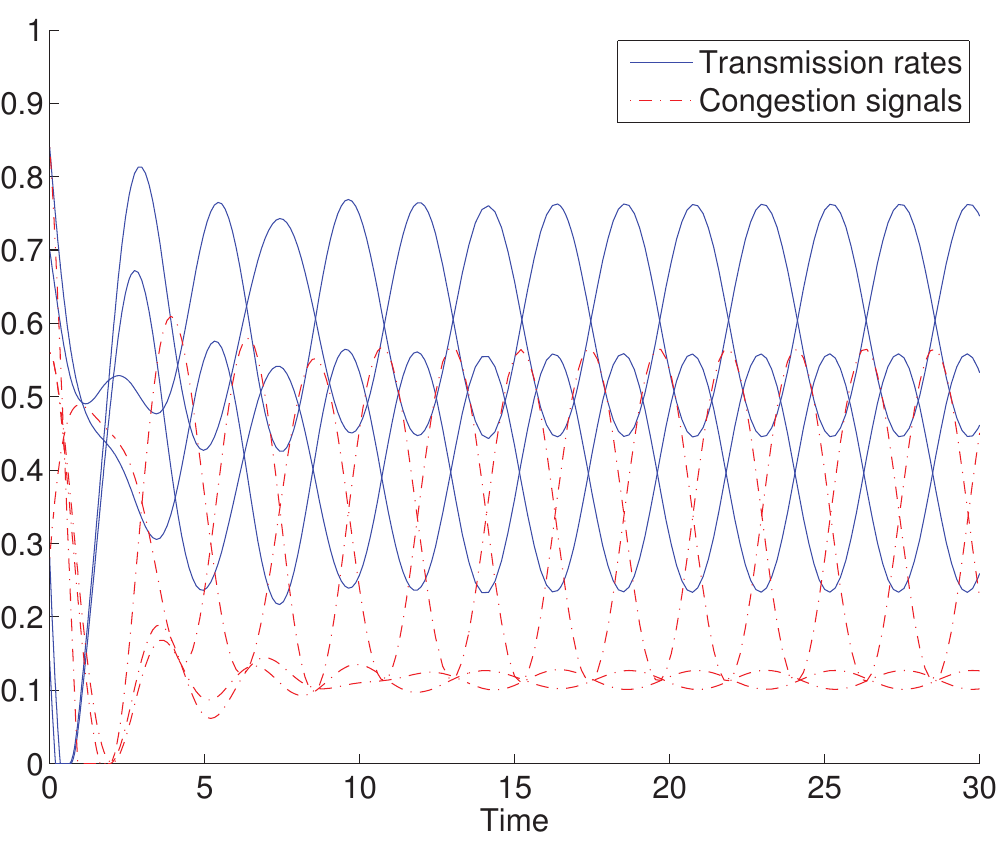}
	\caption{The unmodified dynamics \eqref{eq:multi-path-routing-dynamics} running on the network given in \autoref{fig:network} with all link capacities set to $1$ and the utility functions are $\log(1+x)$ and $1-e^{-x}$ for the sources at $1$ and $2$ respectively. In this network the condition \eqref{eq:algebriac-instability-condition} holds, and there is oscillatory behaviour which does not decay.}
	\label{fig:oscillations1}
\end{figure}
\begin{figure}[!h]
\centering	\includegraphics[keepaspectratio=true,width=0.4\textwidth]{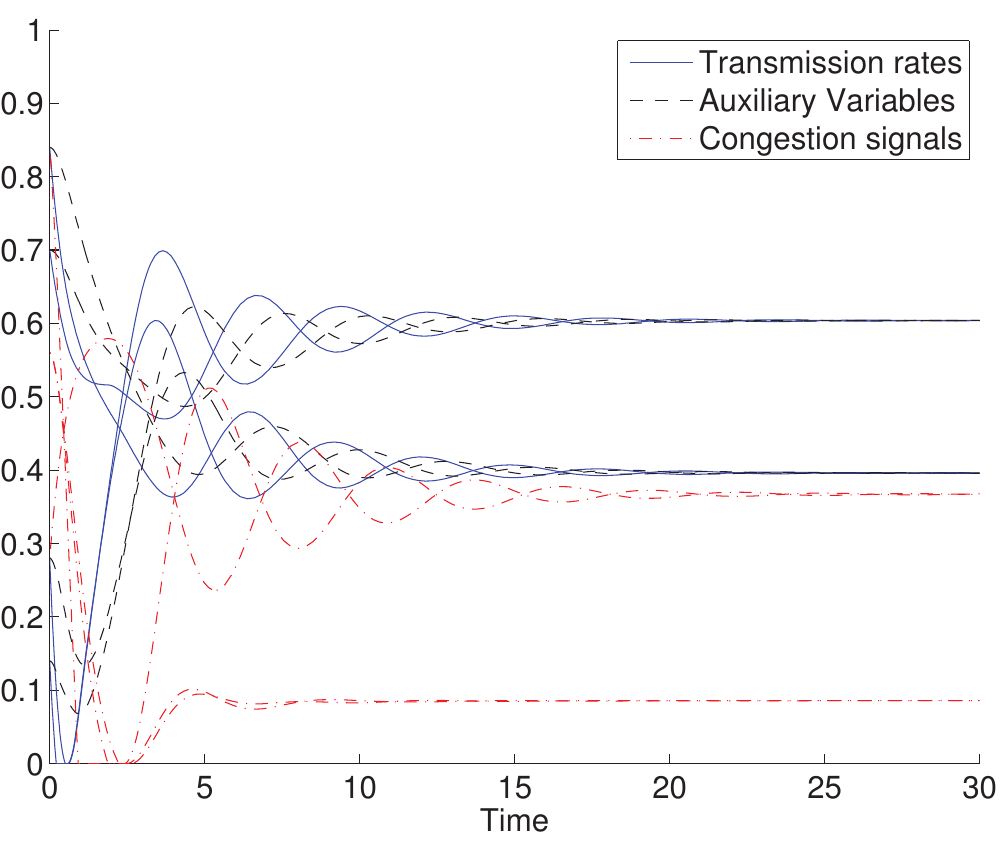}
	\caption{The modified dynamics \eqref{eq:multi-path-routing-modified-dynamics} running on the network given in \autoref{fig:network} with all link capacities set to $1$, $\kappa_j=1$ for all $j$. The utility functions are $\log(1+x)$ and $1-e^{-x}$ for the sources at $1$ and $2$ respectively. In this network the condition \eqref{eq:algebriac-instability-condition} holds, but the modification of the dynamics causes rapid convergence to equilibrium.}
	\label{fig:convergence1}
\end{figure}
\begin{figure}[!h]
	\centering
\resizebox{4.8cm}{2.5cm}{%
	\begin{tikzpicture}[auto,node distance=2cm,
	thick,
	source/.style={circle,fill=blue!20,draw,font=\sffamily\Large\bfseries},
	middle/.style={circle,fill=gray!20,draw,font=\sffamily\Large\bfseries},
	dest/.style={circle,fill=red!20,draw,font=\sffamily\Large\bfseries},
	link1/.style={dotted,->},
	link2/.style={->}]
	\node[source] (1) at (0,0) {$1$};
	\node[middle] (2) [above right of=1] {$2$};
	\node[middle] (3) [right of=2] {$3$};
	\node[middle] (4) [right of=3] {$4$};
	\node[dest] (7) [below right of=4] {$7$};
	\node[middle] (5) [below right of=1] {$5$};
	\node[middle] (6) [below left of=7] {$6$};
	\draw[link2] (1) to [bend left=15] (2);
	\draw[link2] (2) to (3);
	\draw[link2] (3) to (4);
	\draw[link2] (4) to [bend left=15] (7);
	\draw[link2] (1) to [bend right=15] (5);
	\draw[link2] (5) to (6);
	\draw[link2] (6) to [bend right=15] (7);
	\end{tikzpicture}
}
	\caption{A second example network. A single source at $1$ transmits to the destination $7$. It has a choice of two routes.}
	\label{fig:network2}
\end{figure}
\begin{figure}[!h]
\centering	\includegraphics[keepaspectratio=true,width=0.5\textwidth]{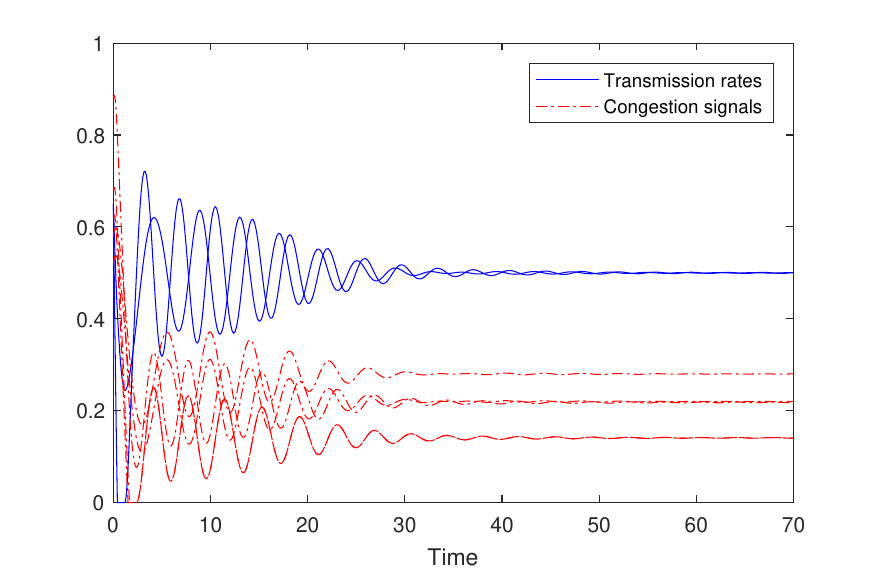}
	\caption{The unmodified dynamics \eqref{eq:multi-path-routing-dynamics} running on the network given in \autoref{fig:network2} with all link capacities set to $0.5$ and the utility function is $\log(1+x)$. The system is asymptotically stable, but displays transient oscillatory behaviour.}
	\label{fig:oscillations2}
\end{figure}
\begin{figure}[!h]
\centering	\includegraphics[keepaspectratio=true,width=0.5\textwidth]{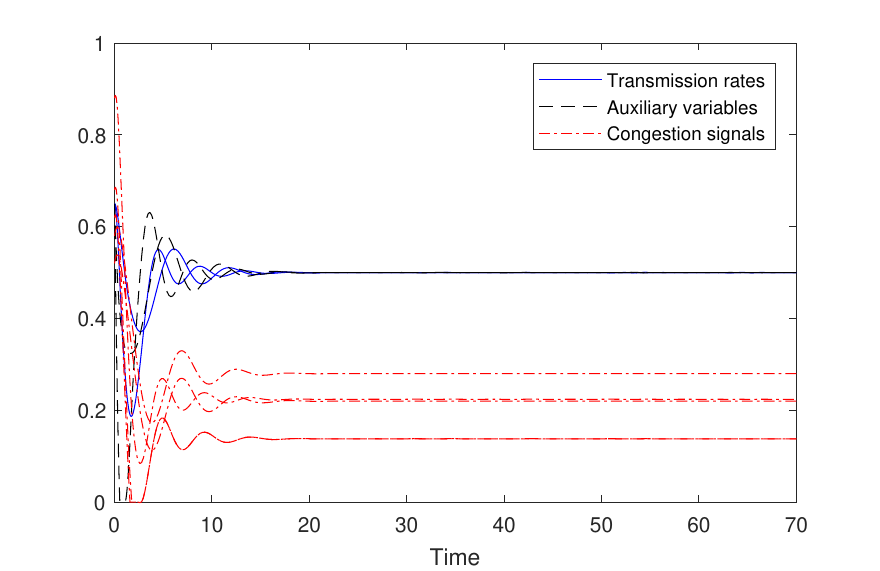}
	\caption{The modified dynamics \eqref{eq:multi-path-routing-modified-dynamics} running on the network given in \autoref{fig:network2} with all link capacities set to $0.5$, $\kappa_j=1$ for all $j$ and the utility function is $\log(1+x)$. The oscillatory behaviour of the unmodified dynamics in \autoref{fig:oscillations2} is damped, and the system rapidly converges to equilibrium.}
	\label{fig:convergence2}
\end{figure}

\subsubsection{Numerical {\color{black}examples}}
\label{sec:numerics}

In this {\color{black}subsection} we present numerical simulations to illustrate {\color{black}the results described above}. We consider the two networks in \autoref{fig:network} and \autoref{fig:network2}.

In \autoref{fig:oscillations1} and \autoref{fig:convergence1} we use the network in \autoref{fig:network} with capacities all set to $1$. The utility functions were chosen as $\log(1+x)$ and $1-e^{-x}$ for the sources at $1$ and $2$ respectively. The parameters $\kappa_j$ were all set to $1$. This network satisfies the condition \eqref{eq:algebriac-instability-condition} and this is apparent in the oscillating modes of the unmodified dynamics \eqref{eq:multi-path-routing-dynamics}, shown in \autoref{fig:oscillations1}, that do not decay. However, when we apply the modified dynamics \eqref{eq:multi-path-routing-modified-dynamics} to this network, we obtain the rapid convergence to the equilibrium shown in \autoref{fig:convergence1}.

In \autoref{fig:oscillations2} and \autoref{fig:convergence2} we use the network in \autoref{fig:network2}. We take the utility function as $\log(1+x)$, and the capacities all set to $0.5$. The parameters $\kappa_j$ were all set to $1$. On this network the original dynamics \eqref{eq:multi-path-routing-dynamics} converge to equilibrium, shown in \autoref{fig:oscillations2}, but there is transient oscillatory behaviour. When we instead implement the modified dynamics \eqref{eq:multi-path-routing-modified-dynamics}, {\color{black}shown in \autoref{fig:convergence2}}, we see an improved performance with more rapid convergence and damping of the oscillations.

\section{Conclusion}
\label{sec:Conclusions}
%{\color{purple}REWRITE}
{\color{black}
In this paper we considered the problem of convergence to a saddle point of a concave convex function via subgradient dynamics that provide a restriction in an arbitrary convex domain. We showed that despite the nonlinear and non-smooth character of these dynamics, \icl{when these have an equilibrium point}  their $\omega$-limit set is comprised of trajectories that are solutions to only linear ODEs.
In particular, we showed that these ODEs are subgradient dynamics on affine subspaces which is a class of dynamics the asymptotic properties of which have been exactly characterized in part I. Various convergence criteria have been deduced from these results that can guarantee convergence to a {\color{black}saddle point.}
% without requiring strict convexity.
Several examples have also been discussed throughout the manuscript to illustrate the results in the paper.
}
%We have considered in the paper the problem of convergence to a saddle point of a general concave-convex function that is not necessarily strictly concave-convex. It has been shown that for all initial conditions the gradient method is guaranteed to converge to a trajectory that satisfies a linear ODE. Extensions have also been given to subgradient methods that constrain the dynamics to a convex domain. Simple characterizations of the limiting solutions have been given in special cases and modified schemes with guaranteed convergence have also been discussed. We have given an example application to multi-path congestion control. Our aim is to further exploit the geometric viewpoint in the paper to investigate improved schemes in this context with higher order dynamics. % in this context.
%%investigate such classes of problems and also use the result in the paper to deduce improves %schemes with guaranteed convergence.

%\vspace{-4mm}
\bibliography{Paper}
\bibliographystyle{plain}

\appendices
\section{Proofs of the main results}
\label{sec:proofs}
In this \icl{appendix} we prove the main results of the paper, which are stated in \autoref{sec:Main} \icl{in the main text}.
\subsection{Outline of the proofs}\label{subsec:Main:sketch}
We first give a brief outline of the derivations of the results to improve {\color{black}their readability.} %of the manuscript.
\subsubsection{Pathwise stability and convex projections}
\icl{In section \ref{sec:isometry} of this appendix we prove the results described in \autoref{subsec:Main:faces} in the main text.}

We revisit some of the literature on topological dynamical systems \cite{Giacomo-equicontinuity}, quoting a more general result \autoref{thm:equicontinuous-implies-flow}, from which \autoref{cor:pathwise-stable-plus-equilibrium-point-implies-isometries} is deduced.
%Then \autoref{Incremental-stability-is-preserved-by-convex-projection} is proved using the %convexity of the domain $K$.
%The combination of
\icl{These results} allow us to prove the main result of the subsection, \autoref{prop:omega-limit-set-in-face}, using the fact that the convex projection term cannot break the isometry property of the flow on the $\omega$-limit set.

\subsubsection{Subgradient method} {\color{black}
In sections \ref{sec:subgradient-method}, \ref{sec:conv_subgradient-method_proof} \icl{in this appendix} we prove the results in subsections \ref{subsec:Main:subgradient-method}, \ref{subsec:Main:gradmethod}, respectively, \icl{in the main text} using the results in \autoref{subsec:Main:faces}.}

\subsection{Convergence to a flow of isometries}
\label{sec:isometry}
%\todo[pictures for this section]
In this section we provide the \icl{proofs of \autoref{cor:pathwise-stable-plus-equilibrium-point-implies-isometries} %, %\autoref{Incremental-stability-is-preserved-by-convex-projection}
and \autoref{prop:omega-limit-set-in-face}}.

We begin by revisiting the literature on topological dynamical systems, in which a type of incremental stability is studied, and show how this leads to an invariance principle for pathwise stability.
\begin{definition}[Equicontinuous semi-flow]
We say that a flow (resp. semi-flow) $(\phi,X,\rho)$ is \emph{equicontinuous} if for any $x(0)\in X$ and $\varepsilon>0$ there is a $\delta=\delta(x(0),\varepsilon)$ such that if $\rho(x'(0),x(0))<\delta$ then
\begin{equation}
\rho(x(t),x'(t))\le \varepsilon \text{ for all }t\in \mathbb{R} \text{ (resp. $\mathbb{R}_+$)}.
\end{equation}
\end{definition}
\begin{remark}
In the control literature equicontinuity of a semi-flow would correspond to `semi-global non-asymptotic incremental stability', but we shall keep the term equicontinuity for brevity and consistency with \cite{Giacomo-equicontinuity}.
\end{remark}

\begin{definition}[Uniformly almost periodic flow]
We say that a flow $(\phi,X,\rho)$ is \emph{uniformly almost periodic} if for any $\varepsilon>0$ there is a \emph{syndetic} set $A\subseteq\mathbb{R}$, (i.e. $\mathbb{R}=A+B$ for some compact set $B\subseteq\mathbb{R}$), for which
\begin{equation}
\rho(\phi(t,x),x)\le\varepsilon\text{ for all }t\in A, x\in X.
\end{equation}
\end{definition}
For the readers convenience we reproduce the results, \cite[Theorem 8]{Giacomo-equicontinuity} and \cite[Proposition 4.4.]{Ellis-lecture-notes}, that we will use.
\begin{theorem}[G. Della Riccia \cite{Giacomo-equicontinuity}]\label{thm:equicontinuous-implies-flow}
Let $(\phi,X,\rho)$ be an equicontinuous semi-flow and let $X$ be either locally compact or complete. Let $\Omega$ be its $\omega$-limit set. Then $(\phi,\Omega,\rho)$ is an equicontinuous semi-flow of homeomorphisms of $\Omega$ onto $\Omega$. This generates an equicontinuous flow.
\end{theorem}
The backwards flow given by \autoref{thm:equicontinuous-implies-flow} is only unique on $\Omega$, (see \autoref{rem:backwards-flow-interpretation} which also applies here).
\begin{proposition}[R. Ellis \cite{Ellis-lecture-notes}]\label{prop:Ellis}
Let $(\phi,X,\rho)$ be a flow, with $X$ compact. Then the following are equivalent:
\begin{enumerate}[(i)]
\item The flow is equicontinuous.
\item The flow is uniformly almost periodic.
\end{enumerate}
\end{proposition}

In our case we study pathwise stability which is a particular form of equicontinuity. We prove stronger results in this special case.
\begin{proof}[Proof of \autoref{cor:pathwise-stable-plus-equilibrium-point-implies-isometries}]
By \autoref{thm:equicontinuous-implies-flow} $(\phi,\Omega,d)$ is an equicontinuous flow with an equilibrium point $\mathbf{\bar{z}}$. Let $R>0$ be arbitrary, and define
\begin{equation}
Y_R=\left\{\mathbf{z}(0)\in \Omega: \sup_{t\in\mathbb{R}}d(\mathbf{z}(t),\mathbf{\bar{z}})\le R\right\}.
\end{equation}
As the flow is equicontinuous, $Y_R$ is a closed bounded subset of $\mathbb{R}^{n+m}$ and hence compact, and moreover, the union of the sets $Y_R$ over $R\ge0$ is $\Omega$. By \autoref{prop:Ellis} the flow $(\phi,Y_R,d)$ is uniformly almost periodic. By pathwise stability, $d:Y_R\times Y_R\to \mathbb{R}$ is a non-increasing along the direct product flow, and is a continuous function on a compact set. Hence we have the inequality, for any two points $\mathbf{z}(0),\mathbf{z}'(0)\in Y_R$,
\begin{equation}
\begin{aligned}
\lim_{t\to-\infty}&d(\mathbf{z}(t),\mathbf{z}'(t))= \sup_{t\in\mathbb{R}}d(\mathbf{z}(t),\mathbf{z}'(t))\\
&\ge \inf_{t\in\mathbb{R}}d(\mathbf{z}(t),\mathbf{z}'(t))= \lim_{t\to\infty}d(\mathbf{z}(t),\mathbf{z}'(t)).
\end{aligned}
\end{equation}
We claim that the two limits are equal. Indeed, by uniform almost periodicity there are sequences $t_n\to\infty$ and $t_n'\to-\infty$ as $n\to\infty$ for which
\begin{equation}
0=\lim_{n\to\infty}d(\mathbf{z}(t_n),\mathbf{z}(0))=\lim_{n\to\infty}d(\mathbf{z}(t_n'),\mathbf{z}(0))
\end{equation}
and the analogous limits hold for $\mathbf{z}'$ for the same sequences $t_n,t_n'$. Hence, by continuity of $d$, we have
\begin{equation}
\lim_{t\to-\infty}d(\mathbf{z}(t),\mathbf{z}'(t))=d(\mathbf{z}(0),\mathbf{z}'(0))=\lim_{t\to\infty}d(\mathbf{z}(t),\mathbf{z}'(t)).
\end{equation}
Hence $d(\mathbf{z}(t),\mathbf{z}'(t))$ is constant. By picking $R$ big enough, this holds for any $\mathbf{z}(0),\mathbf{z}'(0)\in\Omega$, which completes the proof that the sub-semi-flow generates a flow of isometries.

It remains to show that $\Omega$ is convex. To this end let $\mathbf{z}(t),\mathbf{z}'(t)$ be two trajectories of $(\phi,\Omega,d)$. Let that $\lambda\in(0,1)$ and define $\mathbf{z}''(t)=\lambda\mathbf{z}(t)+(1-\lambda)\mathbf{z}'(t)$. By the same argument as used in the proof of \cite[{\color{black}\icl{%\autoref{I-S-is-convex}
Proposition 32}}]
{Holding-Lestas-gradient-method-Part-I} we deduce that $\mathbf{z}''(t)$ is a trajectory of the original semi-flow, but (as argued above) by uniform almost periodicity of $(\phi,\Omega,d)$ we have a sequence of times $t_n\to\infty$ for which $d(\mathbf{z}(t_n),\mathbf{z}(0))\to0$ as $n\to\infty$ and the same limit for $\mathbf{z}'(t)$. Hence $d(\mathbf{z}''(t_n),\mathbf{z}''(0))\to0$ also, showing that $\mathbf{z}''(0)$ is in the $\omega$-limit set.
\end{proof}

%We now work under the set of assumptions \eqref{eq:projected-pathwise-stable-ODE} and consider projected pathwise stable differential equations.
%\begin{proof}[Proof of \autoref{Incremental-stability-is-preserved-by-convex-projection}]
% Let $\mathbf{z}(t)$ and $\mathbf{z}'(t)$ be two arbitrary solutions to the projected ODE, and define  $W(t)=\frac12|\mathbf{z}(t)-\mathbf{z}'(t)|^2$. Then $W$ is absolutely continuous and for almost all times $t\ge0$ we have,
% \begin{equation}\label{W-dot-in-convex-projection-case}
% \begin{aligned}
% \dot{W}(t)&=(\mathbf{z}(t)-\mathbf{z}'(t))^T(\mathbf{\dot{z}}(t)-\mathbf{\dot{z}}'(t))\\
% &=(\mathbf{z}(t)-\mathbf{z}'(t))^T(\mathbf{f}(\mathbf{z}(t))-\mathbf{f}(\mathbf{z}'(t)))+\\
% &\quad-(\mathbf{z}(t)-\mathbf{z}'(t))^T\mathbf{P}_{N_K(\mathbf{z}(t))}(\mathbf{f}(\mathbf{z}(t)))+\\
% &\quad+(\mathbf{z}(t)-\mathbf{z}'(t))^T\mathbf{P}_{N_K(\mathbf{z}'(t))}(\mathbf{f}(\mathbf{z}'(t))).
% \end{aligned}
% \end{equation}
% The first term is non-positive due to the assumption that the {\color{black}ODE satisfies \eqref{eq:projected-pathwise-stable-ODE}.}
% %original
% %was pathwise stable.
% The other two terms are non-positive due to the definition of the normal cone. %\todo[picture]
% \end{proof}
\icl{We now} use the isometry property together with the geometry of the convex projection term to obtain the key result of this section, \autoref{prop:omega-limit-set-in-face}, which states that the limiting dynamics of a pathwise stable ODE restricted to a convex set $K$ have $C^1$ smooth vector field and lie inside one of the faces of $K$.

To prove the theorem we will make use of a simple lemma on faces of convex sets.
\begin{lemma}\label{lem:C-intersects-relint-F}
Let $K\subseteq\mathbb{R}^n$ be non-empty closed and convex and $A\subseteq K$. Let $F$ be the minimal face of $K$ containing $A$, (see \autoref{def:minimal-face}), then $\relint(F)$ intersects $\Conv A$.
\end{lemma}
The statement of this lemma and the idea behind its proof are illustrated by \autoref{fig:C-intersects-relint-F}.
\begin{figure}[h]
\centering
%Start of code
\begin{tikzpicture}[scale=0.6,anchor=mid,>=latex',line join=bevel,]
%\draw[help lines] (0,0) grid +(10,10);
\coordinate (a) at (0,0);
\coordinate (b) at (4,0);
\coordinate (c) at (3,2);
\coordinate [label=right:{$F$}] (F) at ($(b)!0.5!(c)$);
\draw (a) -- (b) -- (c) -- (a);
\coordinate [label=above right:{$A$}](d1) at ($(a)!0.3!(b)$);
\coordinate (d2) at ($(a)!0.5!(b)$);
\coordinate (e1) at ($(a)!0.32!(c)$);
\coordinate (e2) at ($(a)!0.6!(c)$);
\fill [opacity=0.5] (d1) -- (d2) to[out=30,in=0] (e2)--(e1) to[out=210,in=180] (d1);
\end{tikzpicture}
\caption{This figure illustrates the claim of \autoref{lem:C-intersects-relint-F}. The triangle $F$ is the minimal face containing the convex set $A$ (shaded region). If $A$ intersects two subfaces of $F$, then, as shown, to be convex it must also intersect the relative interior of $F$.}\label{fig:C-intersects-relint-F}
\end{figure}

\begin{proof}
As faces are convex, the minimal face containing $A$ is the same as the minimal face containing $\Conv A$. So we are free to assume without loss of generality that $A$ is convex. Assume for a contradiction that $A\cap \relint(F)=\emptyset$. Define the set $\mathcal{F}$ as
%\vspace{-1mm}
\begin{equation*}
\{C:C\text{ is a proper face of }F\text{ and }A\cap (\relint C)\ne \emptyset\}.
\end{equation*}
%\vspace{-1mm}
Note that every point in the relative boundary of $F$ lies in the relative interior of some proper face of $F$ by property (e) below \autoref{def:face}. This implies that $\mathcal{F}$ is not empty. Now, either there is a face $C$ in $\mathcal{F}$ that contains all other faces in $\mathcal{F}$, or there are two faces $F_1,F_2\in \mathcal{F}$ such that there is no face $F_3\in \mathcal{F}$ containing both $F_1$ and $F_2$. In the first case, $C$ is a face containing $A$ that is strictly contained in $F$, contradicting minimality of $F$. In the second case let $x_i\in (\relint F_i)\cap A$ for $i=1,2$, (note that $x_1\ne x_2$ by property (e) of faces), and let $x_3$ be some point in the open line segment between $x_1$ and $x_2$. By convexity of $A$, $x_3\in A$. Hence $x_3$ lies in $\relint(F_3)$ for some face $F_3$, and $F_3\in\mathcal{F}$, as otherwise $x_3$ would lie in $\relint(F)$ contradicting the assumption that $(\relint F)\cap A=\emptyset$. We claim that $F_3$ contains both $F_1$ and $F_2$, a contradiction. Indeed, first we note that $x_1,x_2\in F_3$ by property (ii) in \autoref{def:face} as $x_3\in F_3$. Then, as $F_i$ is convex and $x_i\in\relint (F_i)$, $F_i$ can be written as the union of line segments which have $x_i$ as an interior point (i.e. not an end point). But each of these line segments touches $F_3$ at $x_i$, so by \autoref{def:face}(ii) each lies entirely within $F_3$.
\end{proof}

\icl{
\begin{proof}[Proof of \autoref{prop:converge_Omega}]
From the text above the statement of the Corollary %below %\autoref{Incremental-stability-is-preserved-by-convex-projection}
we have that all trajectories of the semiflow converge to its $\omega$-limit set (denoted as $\Omega$). Also from \autoref{cor:pathwise-stable-plus-equilibrium-point-implies-isometries} we have that $(\phi,\Omega,d)$ defines a flow of isometries.

Convergence to $\Omega$ can now be strengthened to convergence to a solution in $\Omega$ using the same arguments as in the proof of \cite[%\autoref{I-LaSalle-result-full-space}]
Corollary 11]
{Holding-Lestas-gradient-method-Part-I} with set $\mathcal{S}$ replaced with $\Omega$.
\end{proof}
}

\begin{proof}[Proof of \autoref{prop:omega-limit-set-in-face}]
\textbf{Step 1: Identification of the limiting equation.}
First, by \autoref{Incremental-stability-is-preserved-by-convex-projection} and \autoref{cor:pathwise-stable-plus-equilibrium-point-implies-isometries} $(\phi,\Omega,d)$ is a flow of isometries. Now let $F$ be the minimal face that contains $\Omega$, i.e. the intersection of all faces that contain $\Omega$, and $N_F$ be its normal {\color{black}cone (in step 2 of the} proof we will identify this face more precisely). We note that the vector field in \eqref{eq:projected-pathwise-stable-ODE} must be directed parallel to $V$, as otherwise trajectories would leave $F$, contradicting $\Omega\subseteq F$.

It is sufficient to show that if $\mathbf{z}=\mathbf{z}(0)\in\Omega$ with $\mathbf{n}(t)=\mathbf{P}_{N_K(\mathbf{z}(t))}(\mathbf{f}(\mathbf{z}(t)))$ then $\mathbf{n}(t)$ is orthogonal to $F$. If $\mathbf{z}(t)\in\relint K$ then $N_K(\mathbf{z}(t))=N_F$ and the orthogonality holds. Otherwise $\mathbf{z}(t)$ lies in the relative boundary of $F$.

As each solution of the differential equation \eqref{eq:projected-pathwise-stable-ODE} holds only for almost all times $t$ and we wish to consider an uncountably infinite family of solutions, we run the risk of taking an uncountable union of sets of measure zero, (which does not necessarily have zero measure). Avoiding this makes the proof technical. To better communicate the idea of the proof, we shall first give the proof that would work if the differential equations held for all times $t$.

\textbf{Step 1.1: Heuristic (unrigorous) proof.}

Let $C=\Conv\Omega$, then, by the definition of a face, $\Omega\subseteq F$ implies that $C\subseteq F$. From \autoref{lem:C-intersects-relint-F} and the minimality of $F$ we deduce that $C$ must intersect $\relint F$. Thus there are $\mathbf{x}(0),\mathbf{y}(0)\in \Omega$ and $\lambda\in(0,1)$ with $\mathbf{w}=\lambda\mathbf{x}(0)+(1-\lambda)\mathbf{y}(0)\in\relint F$. Set $W=\frac12|\mathbf{x}(t)-\mathbf{z}(t)|^2$. By the isometry property of the flow we know that $\dot{W}=0$ at $t$. \icl{We also have,
 \begin{equation}\label{W-dot-in-convex-projection-case}
 \begin{aligned}
 \dot{W}(t)&=(\mathbf{x}(t)-\mathbf{z}(t))^T(\mathbf{\dot{x}}(t)-\mathbf{\dot{z}}(t))\\
 &=(\mathbf{x}(t)-\mathbf{z}(t))^T(\mathbf{f}(\mathbf{x}(t))-\mathbf{f}(\mathbf{z}(t)))+\\
 &\quad-(\mathbf{x}(t)-\mathbf{z}(t))^T\mathbf{P}_{N_K(\mathbf{x}(t))}(\mathbf{f}(\mathbf{x}(t)))+\\
 &\quad+(\mathbf{x}(t)-\mathbf{z}(t))^T\mathbf{P}_{N_K(\mathbf{z}(t))}(\mathbf{f}(\mathbf{z}(t))).
 \end{aligned}
 \end{equation}
 The first term in \eqref{W-dot-in-convex-projection-case} is non-positive due to the assumption that the {ODE satisfies \eqref{eq:projected-pathwise-stable-ODE}.}
 %original
 %was pathwise stable.
 The other two terms are non-positive due to the definition of the normal cone.
%Following the computation \eqref{W-dot-in-convex-projection-case} in the proof of %\autoref{Incremental-stability-is-preserved-by-convex-projection}
%e deduce
Hence $\dot{W}=0$ implies}
that $(\mathbf{x}-\mathbf{z})^T\mathbf{n}=0$. Similarly we obtain $(\mathbf{y}-\mathbf{z})^T\mathbf{n}=0$. Taking a convex combination of these equalities, we obtain
\begin{equation}
  (\mathbf{w}-\mathbf{z})^T\mathbf{n}=\lambda(\mathbf{x}-\mathbf{z})^T\mathbf{n}+(1-\lambda)(\mathbf{y}-\mathbf{z})^T\mathbf{n}=0+0=0
  \end{equation}
  and as $\mathbf{w}$ is in the relative interior of $F$ this implies that $\mathbf{n}$ is orthogonal to $F$. %\todo[picture]

\textbf{Step 1.2: Rigorous proof.}
We now give the fully rigorous proof. We must show that the set of times $t$ when $\mathbf{n}(t)$ is not orthogonal to $F$ is of measure zero. Let $\Omega'$ be a countable dense subset of $\Omega$ that contains $\mathbf{z}(0)$. By invariance of $\Omega$ under the flow $\phi$, the set $\phi(t,\Omega')=\{\phi(t,\mathbf{x}):\mathbf{x}\in\Omega'\}$ is also dense in $\Omega$ for any $t\in\mathbb{R}$. Then the set
  \begin{equation}
  \begin{aligned}
    A=\{t&\in[0,\infty):\exists \mathbf{x}(0)\in \Omega' \text{ such that }\\
       &\dot{\mathbf{x}}(t)\ne \mathbf{f}(\mathbf{x}(t))-\mathbf{P}_{N_K(\mathbf{x}(t))}(\mathbf{f}(\mathbf{x}(t)))\}
  \end{aligned}
  \end{equation}
  is the countable union of measure zero sets, and is hence of measure zero. From the isometry property and by considering $W(t)=\frac12|\mathbf{x}(t)-\mathbf{z}(t)|^2$ with $\mathbf{x}(0)\in\Omega'$, it follows that $(\mathbf{x}(t)-\mathbf{z}(t))^T\mathbf{n}(t)=0$ for all $\mathbf{x}(0)\in\Omega'$ and $t\in[0,\infty)\setminus A$. Thus, for $t\in[0,\infty)\setminus A$, $(\mathbf{x}-\mathbf{z}(t))^T\mathbf{n}(t)=0$ for all $\mathbf{x}$ in a dense subset of $\Omega$, and hence for any $\mathbf{x}\in\Omega$. The proof now follows as step 1.1. above.

  \textbf{Step 2: Identification of the limiting face.}
Finally we will show that the face $F$ defined above is in fact the minimal face $F'$ containing the equilibrium points of the semi-flow $(\phi,K,d)$. We argue by contradiction. If $F\ne F'$ then there must be some trajectory $\mathbf{z}(t)$ in $\Omega$ and a time $t_0$ with $\mathbf{z}(t_0)\in F\setminus F'$. For $T>0$ we define $\mathbf{z}(t;T)=\frac1{2T}\int^T_{-T}\mathbf{z}(t+s)\,ds$. For any finite $T$ this is a convex combination of trajectories in $\Omega$, and as $\Omega$ is convex by \autoref{cor:pathwise-stable-plus-equilibrium-point-implies-isometries}, $t\mapsto \mathbf{z}(t;T)$ is a trajectory in $\Omega$. Next, as the semi-flow is uniformly almost periodic due to \autoref{prop:Ellis} the trajectory $\mathbf{z}(t)$ is an almost periodic function. Therefore, the limit $T\to\infty$ of $\mathbf{z}(t;T)$ exists (see e.g. \cite{Ellis-lecture-notes}), and this limit is clearly a constant ($\mathbf{z}'$ say) independent of $t$. As $\Omega$ is closed, $\mathbf{z}'\in \Omega$ and being a constant, is an equilibrium point of the semi-flow.

To obtain a contradiction we argue that $\mathbf{z}'\not\in F'$ which is impossible as $F'$ contains all equilibrium points. Indeed, this follows as the trajectory $\mathbf{z}(t)$, being almost periodic and passing through $\mathbf{z}(t_0)\in F\setminus F'$ spends a positive proportion of its time in $F\setminus F'$. Therefore, there is a $\delta>0$ such that for any sufficiently large $T$, the average $\mathbf{z}(t;T)$ satisfies $d(\mathbf{z}(t;T),F)\ge\delta$ and this property carries over to the limit~$\mathbf{z}'$.
  \end{proof}

\subsection{Subgradient method}\label{sec:subgradient-method}
In this section we give the proofs of the results of \autoref{subsec:Main:subgradient-method}.

\begin{proof}[Proof of \autoref{thm:subgradient-method-faces}]
We apply \autoref{prop:omega-limit-set-in-face}, {\color{black}noting that \icl{$\mathbf{f}(\mathbf{z})$ in \eqref{gradmethod-convex-domain} satisfies the inequality in \eqref{eq:projected-pathwise-stable-ODE} \cite{rockafellar1970convex}, \cite{goebel2017stability}.}
% holds from the pathwise stability property of the gradient method \cite[{\color{black} Proposition 9}]{Holding-Lestas-gradient-method-Part-I}.
%gradient method is pathwise stable by \cite[{\color{black} Proposition %9}]{Holding-Lestas-gradient-method-Part-I}, {\color{black}and hence . Let $F$ be the minimal %face given by \autoref{prop:omega-limit-set-in-face}. There are two cases.

\textbf{Case (i).} % $\bar{\mathcal{S}}\cap \interior K$ is empty.}
%We are in case (ii) of the theorem.
This follows directly from \autoref{prop:omega-limit-set-in-face}.

\textbf{Case (ii).} %$\bar{\mathcal{S}}\cap \interior K$ is non-empty.}
As $F$ must contain all $K$-restricted saddle points, it must contain a point in the interior of $K$. The only such face is $K$ itself whose affine span is $\mathbb{R}^{n+m}$ (as $K$ has non-empty interior) which has normal cone $\{\mathbf{0}\}$. Therefore in case (ii) \eqref{eq:projected-ODE-on-face} becomes the gradient method \eqref{gradmethod-fullspace} and \eqref{eq:omega-is-S-cap-K} holds.

The convexity and isometry properties of $\Omega$ stated follow from \autoref{cor:pathwise-stable-plus-equilibrium-point-implies-isometries}.
%\textbf{Case 2. $\bar{\mathcal{S}}\cap \interior K$ is empty.} We are in case (ii) of the %theorem. The claims of (ii) follow directly from \autoref{prop:omega-limit-set-in-face}.
}
 \end{proof}

\subsection{A general convergence criterion}\label{sec:conv_subgradient-method_proof}
In this section we give the proofs of \autoref{subsec:Main:gradmethod}.
\begin{proof}[Proof of \autoref{thm:subgrad-method-final-convergence-criterion}]
By \autoref{thm:subgradient-method-faces}(i) any solution $\mathbf{z}(t)$ in the $\omega$-limit set of the subgradient method on $K$ solves \eqref{eq:subgradient-method-limiting-solutions}. By using $\mathbf{\Pi}$, the orthogonal projection matrix onto the orthogonal complement of $N_V$, the ODE \eqref{eq:subgradient-method-limiting-solutions} can be written as \eqref{eq:projected-gradient-method}. {\color{black}Noting also the isometry property of the $\omega$-limit set, we have by} \autoref{thm:projected-gradient-method-result} (in {\color{black}Appendix \ref{app:projected-gradient-method}}), {\color{black}that} $\mathbf{z}(t)$ satisfies \eqref{eq:projected-gradient-method-linear-ODE} and \eqref{eq:first-thm-kernel-condition-projected-gradient-method} for all $t\in\mathbb{R}$ and $r\in[0,1]$. Therefore, if there are no non-constant trajectories of the subgradient method on $K$ satisfying these conditions then the $\omega$-limit set consists only of equilibrium points and the subgradient method on $K$ is globally convergent.
\end{proof}
{\color{black}
\begin{proof}[Proof of \autoref{cor:subgradient_linear}]
This follows from \autoref{thm:subgradient-method-faces} and  \autoref{thm:projected-gradient-method-result} using the arguments in the proof of \autoref{thm:subgrad-method-final-convergence-criterion}.
\end{proof}}
% \autoref{cor:subgradient-method-face-criterion} and \autoref{cor:convenience} are now simple consequences of this result.
% \begin{proof}[Proof of \autoref{cor:universal-convergence-criterion}]
% By \autoref{cor:subgradient-method-face-criterion}, (a) holds if and only if the subgradient method is globally convergent on any affine space $K$. Then by the relationship between the subgradient method on an affine space and the projected gradient method, this holds if and only if the projected gradient method is globally convergent for any choice of orthogonal projection matrix $\mathbf{\Pi}$.
%\end{proof}
%
%\begin{proof}[Proof of \autoref{cor:convenience}]
%We apply \autoref{thm:subgradient-method-faces}. From this we obtain a (possibly proper) face $F$ of $K$ containing $\mathbf{0}$ and the $\omega$-limit set $\Omega$. Let $\mathbf{\Pi}$ be the orthogonal projection matrix onto the affine span of $F$. By \autoref{thm:subgradient-method-faces} any limiting solutions of the subgradient method on $K$ solve \eqref{eq:projected-gradient-method} for this choice of $\mathbf{\Pi}$, and by \autoref{thm:projected-gradient-method-result} such solutions satisfy \eqref{eq:projected-gradient-method-linear-ODE}-\eqref{eq:first-thm-kernel-condition-projected-gradient-method} for all $t\in\mathbb{R}$ and $r\in[0,1]$. But by the assumptions of the corollary, the only such solutions are equilibrium points, and hence $\Omega$ consists only of equilibrium points and the subgradient method on $K$ is globally convergent.
%\end{proof}

\section{\color{black}Proofs of the results in \autoref{sec:modification-methods}}\label{sec:proofs-of-examples}
 \subsection{Modification methods}
%\begin{proof}[
{\em Proof of \autoref{thm:convergence-of-modification-methods}:}\\
 {\color{black}We prove convergence of each modification method in turn.} %consider
 %each method in turn.}
 %\autoref{thm:convergence-of-modification-methods} may then be obtained by combining all the %results proved in each subsection below.
 \subsubsection{Auxiliary variables method}
 {\color{black}
  \begin{proposition}\label{cor:auxiliary-convergence}
 Let \eqref{modified-varphi} hold, and assume that there exists a $K'$-restricted saddle point. Then the subgradient method \eqref{gradmethod-convex-domain} on $K'$ {\color{black}applied to $\varphi'$} is globally convergent.
 \end{proposition}
 \begin{proof}
 We prove global convergence to an equilibrium point by making use of Theorem \ref{thm:subgrad-method-final-convergence-criterion}. In particular, we show that the only solutions of the subgradient method applied to $\phi'$, which satisfy both \eqref{eq:projected-gradient-method-linear-ODE} and \eqref{eq:first-thm-kernel-condition-projected-gradient-method}, are equilibrium points.

 Without loss of generality, we assume, by a translation of coordinates, that
 $\mathbf{\bar{z}}'=(M\bar{x},\bar{x},\bar{y})=\mathbf{0}$ is an equilibrium point.
 Since the auxiliary variables are unconstrained the orthogonal complement of $N_V$ in Theorem \ref{thm:subgrad-method-final-convergence-criterion} is a subspace of the form $\mathbb{R}^{n'}\times V'$ where $V'\subseteq\mathbb{R}^{n+m}$ is an affine subspace.

 Let $\mathbf{\Pi}$ be the orthogonal projection matrix onto the subspace $\mathbb{R}^{n'}\times V'$. We decompose $\mathbf{\Pi}$ on $\mathbb{R}^{n'}\times\mathbb{R}^{n+m}$ as
  \begin{equation}
  \mathbf{\Pi}=\begin{bmatrix}
  I&\mathbf{0}\\
  \mathbf{0}&\mathbf{\Pi}'
  \end{bmatrix}.
  \end{equation}
  Now let $\mathbf{z}(t)=(x'(t),x(t),y(t))$ be a solution of the modified subgradient method that satisfies \eqref{eq:projected-gradient-method-linear-ODE} and \eqref{eq:first-thm-kernel-condition-projected-gradient-method}, and let $(\tilde{x}(t),\tilde{y}(t))=\mathbf{\Pi}'(x(t),y(t))$.
The remainder of the proof is carried out in three steps.

  {\bfseries\boldmath Step 1: $x'(t)$ is constant.}
  %By applying %\autoref{thm:projected-gradient-method-result}
  %(noting \autoref{rem:localisation-projected-gradient-method})
  %we have that $\mathbf{z}$ solves \eqref{eq:projected-gradient-method-linear-ODE}.
  By the form of $\mathbf{A}(\mathbf{0})$ in \eqref{eq:projected-gradient-method-linear-ODE}
  %in Theorem \ref{thm:subgrad-method-final-convergence-criterion}
  we deduce that $\dot{x}'(t)=0$.

  {\bfseries\boldmath Step 2: $ \tilde{x}(t)$ and $\tilde{y}(t)$ are constant.}
  From the condition \eqref{eq:first-thm-kernel-condition-projected-gradient-method}
  that $\icl{\mathbf{\Pi}}\mathbf{B}(r\mathbf{z})\mathbf{\Pi}\mathbf{z}=0$ for $r\in[0,1]$, we have that
  \begin{equation}
  0=\mathbf{z}^T\mathbf{\Pi}\mathbf{B}(r\mathbf{z})\mathbf{\Pi}\mathbf{z}=u^T\psi_{uu}u+\tilde{x}^T\varphi_{xx}\tilde{x}-\tilde{y}^T\varphi_{yy}\tilde{y}
  \end{equation}
  where $\psi_{uu}$ is the Hessian matrix of $\psi$ evaluated at $u=M\tilde{x}-x'$. As each term is non-positive and $\psi$ is strictly concave we deduce that $M\tilde{x}-x'=0$ and $\tilde{x}\in\ker(\varphi_{xx}(\mathbf{0}))$.  Thus $M\tilde{x}(t)$ is constant. By the condition that $\ker(M)\cap\ker(\varphi_{xx})=\{0\}$ we deduce that $\tilde{x}(t)$ is constant. Then the form of $\mathbf{A}(\mathbf{0})$ allows us to deduce that $\tilde{y}(t)$ is also constant.

  {\bfseries\boldmath Step 3: $x(t)$ and $y(t)$ are constant.}
  The vector field in \eqref{eq:projected-gradient-method-linear-ODE} is orthogonal to $\ker(\mathbf{\Pi})$, so that $(\tilde{x}(t),\tilde{y}(t))$ being constant implies that $(x(t),y(t))$ are constant.

This completes the proof of convergence to an equilibrium point of the subgradient method applied to $\phi'$ .
\end{proof}
  %This proves that any limiting solution to the subgradient method on $V$ is an equilibrium point, and therefore that the subgradient method on $V$ is globally convergent.
 }

 %The majority of the proof for the auxiliary variables method was completed in \cite{Holding-Lestas-gradient-method-Part-I}. We provide the remainder below.
% \begin{proof}[Proof of \autoref{thm:convergence-of-modification-methods}.1)]
% It was proved in \cite[{\color{black}modifcation method convergence}]{Holding-Lestas-gradient-method-Part-I} that for any affine subspace $V$ containing a $V$-restricted saddle point the subgradient method applied to $\varphi'$ is globally convergent. By \autoref{thm:subgradient-method-faces} this implies global convergence of the subgradient method for any closed convex $K$.
% \end{proof}
 \subsubsection{Penalty {\color{black}function method}}

 %The proof below shows that this method gives convergence by adding just enough strict convexity to eliminate the oscillations caused by the matrix $\mathbf{A}(\mathbf{0})$.
 \begin{proposition}\label{cor:penalty-function-method-converges-convex-domain}
 Let $K\subseteq\mathbb{R}^{n+m}$ be non-empty closed and convex \icl{as in \eqref{set:K}}. Let \eqref{lag-conds-fullspace}, \eqref{penalty-function-method} hold, and assume that there exists a $K$-restricted saddle point. Then the subgradient method \eqref{gradmethod-convex-domain} on $K$ {\color{black}applied to $\varphi'$} is globally convergent.
 \end{proposition}
 \begin{proof}
% By the change of variables
% \begin{equation*}
% \begin{aligned}
% &\tilde{\varphi}'(x,y)=\varphi'(x+\bar{x},y+\bar{y})\\
% &\,=U(x+\bar{x})+\bar{y}^Tg(x+\bar{x})+y^Tg(x+\bar{x})+\psi(g(x+\bar{x}))\\
% &\,=\tilde{U}(x)+y^T\tilde{g}(x)+\psi(\tilde{g}(x))
% \end{aligned}
% \end{equation*}
% for a $K$-restricted saddle point $(\bar{x},\bar{y})$
{\color{black} Without loss of generality, we may assume by a translation of coordinates}
 %we may assume
 that $\mathbf{0}$ is a $K$-restricted saddle point. We apply \autoref{thm:subgrad-method-final-convergence-criterion} and let $F,V,\mathbf{\Pi}$ be as in \autoref{thm:subgrad-method-final-convergence-criterion} and $\mathbf{z}(t)=(x(t),y(t))$ be a trajectory of the subgradient method on $K$ satisfying \eqref{eq:projected-gradient-method-linear-ODE} and \eqref{eq:first-thm-kernel-condition-projected-gradient-method} for all $t\in\mathbb{R}$ and $r\in[0,1]${\color{black}. Define} $(\tilde{x}(t),\tilde{y}(t))=\tilde{\mathbf{z}}(t)=\mathbf{\Pi}\mathbf{z}(t)$. We compute that
 \begin{equation}
 \mathbf{A}(\mathbf{0})=\begin{bmatrix}
 0&g_x(0)^T\\
 -g_x(0)&0
 \end{bmatrix}.
 \end{equation}

 {\bfseries\boldmath Step 1: $g_x(0)\tilde{x}(t)=0$.}

  The condition \eqref{eq:first-thm-kernel-condition-projected-gradient-method} implies that the following expression is zero for all $s\in[0,1]$,
  \begin{equation}\label{eq:cor:penalty-function-method-converges-convex-domain-1}
  \tilde{\mathbf{z}}^T\mathbf{B}(s\mathbf{z}){\color{black}\tilde{\mathbf{z}}}
  =\tilde{x}^T\varphi_{xx}\tilde{x}+[g_x\tilde{x}]^T\psi_{uu}[g_x\tilde{x}]+\psi_u(\tilde{x}^Tg_{xx}\tilde{x})
  \end{equation}
  where $\varphi_{xx}$ is evaluated at $s\mathbf{z}$, with $g_x,g_{xx}$ at $sx$, and $\psi_{uu},\psi_{u^k}$ at $u=g(sx)$, and where $x^Tg_{xx}x$ is the vector with $i$th component $x^Tg^i_{xx}x$ where $g=[g^1,\dotsc ,g^m]^T$. All the terms are non-positive by the assumptions on $\psi$ and $\varphi$. Strict concavity of $\psi$ and that \eqref{eq:cor:penalty-function-method-converges-convex-domain-1} vanishes for all $s\in[0,1]$ implies that $g_x(sx)\tilde{x}=0$ for all $s\in[0,1]$. In particular $g_x(0)\tilde{x}(t)=0$.

  {\bfseries\boldmath Step 2: $\tilde{x}(t)$ is constant.}

%  Let $\mathbf{\Pi}$ be decomposed as in \eqref{eq:Pi-decomposition}.

\icl{  Let $\mathbf{\Pi}$ be decomposed on $\mathbb{R}^n\times\mathbb{R}^m$ as
 \begin{equation}\label{eq:Pi-decomposition}
 \mathbf{\Pi}=\begin{bmatrix}
 \Pi_{11}&\Pi_{12}\\
 \Pi_{21}&\Pi_{22}
 \end{bmatrix}.
 \end{equation}}
  Then $\tilde{x},\tilde{y}$ satisfy
  \begin{align}\label{eq:penalty_tilde}
  \dot{\tilde{x}}&=\Pi_{11}g_x(0)^T\tilde{y}&\quad \dot{\tilde{y}}&=-\Pi_{21}g_x(0)^T\tilde{y}.
  \end{align}

  Taking the time derivative of $g_x(0)\tilde{x}=0$ we obtain $g_x(0)\Pi_{11}g_x(0)^T\tilde{y}=0$. As $\Pi_{11}$ is positive semi-definite, $\ker(g_x(0)\Pi_{11}g_x(0)^T)=\ker(\Pi_{11}g_x(0)^T)$, and hence $\dot{\tilde{x}}=\Pi_{11}g_x(0)^T\tilde{y}=0$ and $\tilde{x}(t)$ is constant.

  {\bfseries\boldmath Step 3: $\tilde{y}(t)$ is constant.}
%  \icl{Note that $\mathbf{\Pi}$ is an orthogonal projection matrix to $V=\text{aff}(F)$ where $F$ is a face of set $K$, with $K$ being the cartesian product of two non-empty convex sets in $\mathbb{R}^n$, $\mathbb{R}^m$ respectively. Due to the product structure of $K$, $V$ must also decompose into $V=V_x\times V_y$ with $V_x\subseteq\mathbb{R}^{n}$ and $V_y\subseteq\mathbb{R}^m$ affine subspaces (note also that $V_x, V_y$ are nonempty since $(0,0)\in V_x\times V_y$). Hence $\Pi_{12}=\Pi_{21}=0$. Therefore $\dot{\tilde y}=0$ from \eqref{eq:penalty_tilde}, which implies $\tilde y$ is constant.}

 The relation $\mathbf{\Pi}\dot{\tilde{\mathbf{z}}}=\dot{\tilde{\mathbf{z}}}$ implies that $\Pi_{11}\dot{\tilde{x}}+{\color{black}\Pi_{12}}\dot{\tilde{y}}=\dot{\tilde{x}}=0$ and $0=\Pi_{12}\dot{\tilde{y}}=-\Pi_{12}\Pi_{21}g_x(0)^T\tilde{y}$. Therefore, again, as $\Pi_{12}\Pi_{21}$ is positive semi-definite we have $\tilde{y}^Tg_x(0)\Pi_{12}\Pi_{21}g_x(0)^T\tilde{y}=0$ and $\Pi_{21}g_x(0)^T\tilde{y}=0=-\dot{\tilde{y}}$, which implies $\tilde{y}$ is constant\footnote{\icl{Note that step 3 could also be proved from the fact that the product structure of $K$ implies that $V=\text{aff}(F)$ must also decompose into $V=V_x\times V_y$ with $V_x\subseteq\mathbb{R}^{n}$, $V_y\subseteq\mathbb{R}^m$ affine subspaces, thus implying  $\Pi_{12}=\Pi_{21}=0$ (this structure of $V$ is used in the proof of \autoref{cor:contraint_mod}).}}.

\icl{The fact that $x(t)$, $y(t)$ are constant can be deduced as in Step 3 of the proof of \autoref{cor:auxiliary-convergence}.}
 \end{proof}

 \subsubsection{Constraint modification method}
 We first consider the case without constraints. The proof below shows that the method works by disrupting the linear structure of the oscillating solutions by changing $\mathbf{A}(\mathbf{z})$ to ensure it is not equal to $\mathbf{A}(\mathbf{0})$, (where $\mathbf{0}$ is a saddle).

%\balance
 \begin{proposition}\label{prop:constraint-modification-method}
 Let \eqref{constraint-modification-method-assumptions} hold and $\bar{\mathcal{S}}\ne\emptyset$. Then $\mathcal{S}=\bar{\mathcal{S}}$ and the gradient method \eqref{gradmethod-fullspace} {\color{black}applied to $\varphi'$} is globally convergent.
 \end{proposition}
 \begin{proof}
 Without loss of generality we may assume that $\mathbf{0}$ is a saddle point of $\varphi$.
 We use the classification of $\mathcal{S}$ given \icl{by \cite[%\autoref{I-Classification-in-terms-of-B(z)z}]
 Theorem 12]
 {Holding-Lestas-gradient-method-Part-I}} and use the notation therein. We first compute,
 \begin{equation}
  \mathbf{A}(\mathbf{z})=
 \begin{bmatrix}
  0& (\psi_gg_x)^T\\
 -\psi_gg_x&0
 \end{bmatrix}.
 %\begin{bmatrix}
  %                        0&\sum_{i=1}^m\psi^i_g(g^i(x))g^i_x(x)^T\\
  %			-\sum_{i=1}^m\psi^i_g(g^i(x))g^i_x(x)&0
   %                      \end{bmatrix}
 \end{equation}
{\color{black} Let $\mathbf{z}(t)=(x(t),y(t))\in\mathcal{S}$} then we have
 \begin{equation} %\label{equation-for-constraint-modification-proof}
 \begin{aligned}
 0&=\frac{d}{ds}[(\psi_g^i(g(sx))^Tg_x(sx)x]_{s=0}\text{ for }i=1,\dotsc,m
 %0&=\frac{d}{ds}\left[\sum_{i=1}^m\psi^i_g(g^i(sx))g^i_x(sx)\right]_{s=0}x\\
 %&=\sum_{i=1}^m(g^i_x(0)x)^T\psi_{gg}^i(0)g_x^i(0)x+\psi_g^i(0)x^Tg_{xx}^i(0)x
 \end{aligned}
 \end{equation}
 Then by applying the chain rule we obtain
 \begin{equation}
 0=[g_x(0)x]^T\psi^i_{gg}(0)[g_x(0)x]+\psi^i_g(0)^T(x^Tg_{xx}(0)x),
 \end{equation}
 where $x^Tg_{xx}(0)x$ is the vector with components $x^Tg^i_{xx}x$ where $g=[g^1,\dotsc g^m]^T$.
 All the terms are non-positive due to the assumptions on $\psi$ and $g$. As $\psi_{gg}^i<0$ we have $g_x(0)x=0$. Hence $\dot{y}=0$ and {\color{black}therefore $y$} is constant. As $|x|^2+|y|^2$ is also constant this means that $\dot{x}$ is zero. {\color{black}Therefore $\mathcal{S}=\bar{\mathcal{S}}$} and the gradient method is globally convergent.
 \end{proof}
 Now we extend the stability to the subgradient method on sets which have a product structure,
 {\color{black}by making use of \autoref{cor:subgradient-method-face-criterion}.}
% . Due to \autoref{cor:subgradient-method-face-criterion} this is essentially an exercise in %algebra.
 \begin{corollary}\label{cor:contraint_mod}
 \icl{Let $K\subseteq\mathbb{R}^{n+m}$ be non-empty closed and convex \icl{as in \eqref{set:K}}.}
 %Let $K=K_x\times K_y$ for $K_x\subseteq \mathbb{R}^n$ and $K_y\subseteq \mathbb{R}^m$ non-empty closed and convex.
 Let \eqref{constraint-modification-method-assumptions} hold and there be a $K$-restricted saddle point. Then the subgradient method \eqref{gradmethod-convex-domain} on $K$ {\color{black}applied to $\varphi'$} is globally convergent.
 \end{corollary}
 \begin{proof}
 By \autoref{cor:subgradient-method-face-criterion} it suffices to prove that the subgradient method on $\affinespan(F)$ \icl{is globally convergent,} where $F$ is an arbitrary face of $K$ that contains a $K$-restricted saddle point $\mathbf{\bar{z}}$. By translation of coordinates we may assume that $\mathbf{\bar{z}}=\mathbf{0}$. By the product structure of $K$, $V=\affinespan(F)$ must also decompose into $V=V_x\times V_y$ with $V_x\subseteq\mathbb{R}^{n}$ and $V_y\subseteq\mathbb{R}^m$ affine subspaces.
 %, \icl{as also noted in Step 3 of the proof of \autoref{cor:penalty-function-method-converges-convex-domain} .}.
 Let the orthogonal projection matrices onto $V_x,V_y$, which exist as $(0,0)\in V_x\times V_y$, be $P,Q$ respectively. Then the subgradient method on $V$, satisfies, for $(x,y)\in V$,
 \begin{equation}
 \begin{aligned}
 \dot{x}&=P\varphi'_x=\varphi^V_x,&\dot{y}&=-Q\varphi'_y=-\varphi^V_y
 \end{aligned}
 \end{equation}
 where $\varphi^V(x,y):=\varphi(Px,Qy)$.
 %\icl{Note that a rotation of coordinate bases does not affect the concavity of $U$ and $g$ in %\eqref{constraint-modification-method-assumptions}}
 By a rotation\footnote{\icl{Note that a rotation of coordinates will transform $\phi'$ in  \eqref{constraint-modification-method-assumptions} to a function that is still of the form specified in  \eqref{constraint-modification-method-assumptions}, i.e. in the new coordinates $\phi'$ can be written in terms of functions $U$, $g$, $\psi$ that satisfy the conditions in  \eqref{constraint-modification-method-assumptions}.}} of coordinate bases we may assume that $V_x=\mathbb{R}^{n'}\times\{0\}$ and $V_y=\mathbb{R}^{m'}\times\{0\}$ for some $n'\le n$ and $m'\le m$. Then $\varphi^V:\mathbb{R}^{n'}\times\mathbb{R}^{m'}\to\mathbb{R}$ is of the form \eqref{constraint-modification-method-assumptions} and \autoref{prop:constraint-modification-method} gives convergence.
 \end{proof}
 %\end{proof}

% \subsection{Multi-path congestion control}
% \begin{proof}[Proof of \autoref{prop:algebriac-instability}]
% The \emph{if} claim follows directly from the discussion {\color{black}preceding the proposition}. For the \emph{only if} we explicitly construct a trajectory that does not converge. Let $u$ satisfy \eqref{eq:algebriac-instability-condition}, then it can be directly verified that
% \begin{equation*}
% \mathbf{z}(t)=\bar{\mathbf{z}}+ce^{t\mathbf{A}(\bar{\mathbf{z}})}\begin{bmatrix}
% u\\-Au
% \end{bmatrix}
% \end{equation*}
% is a solution (for any $c>0$) of the unconstrained gradient method \eqref{gradmethod-fullspace} applied to $\varphi$. By taking $c$ small enough using that $\bar{\mathbf{z}}>0$ (and skew-symmetry of $\mathbf{A}(\bar{\mathbf{z}})$) we can ensure that $\mathbf{z}(t)>0$ for all $t\in\mathbb{R}$, and hence $\mathbf{z}(t)$ is also a solution of the subgradient dynamics \eqref{eq:multi-path-routing-dynamics}.
% \end{proof}

\icl{
\subsection{Multi-path congestion control}
 \begin{proof} [Proof of \autoref{prop:algebriac-instability}]
 The \emph{if} claim follows directly from the discussion {preceding the proposition}. For the \emph{only if} we explicitly construct a trajectory that does not converge. Let $u$ satisfy \eqref{eq:algebriac-instability-condition}, then it can be directly verified that
 \begin{equation*}
 \mathbf{z}(t)=\bar{\mathbf{z}}+ce^{t\mathbf{A}(\bar{\mathbf{z}})}\begin{bmatrix}
 u\\-Au
 \end{bmatrix}
 \end{equation*}
 is a solution (for any $c>0$) of the unconstrained gradient method \eqref{gradmethod-fullspace} applied to $\varphi$. By taking $c$ small enough using the fact that $\bar{\mathbf{z}}>0$ (and the skew-symmetry of $\mathbf{A}(\bar{\mathbf{z}})$) we can ensure that $\mathbf{z}(t)>0$ for all $t\in\mathbb{R}$, and hence $\mathbf{z}(t)$ is also a solution of the subgradient dynamics \eqref{eq:multi-path-routing-dynamics}.
 \end{proof}
}

\end{document}